\RequirePackage{fix-cm}
\documentclass[smallextended]{svjour3}       
\smartqed  
    \usepackage{graphicx}
    \usepackage{amsmath}
    \usepackage{epstopdf}
    \usepackage{amsfonts}
    \usepackage{dsfont}
    \usepackage{amssymb}
    \usepackage{blindtext}
    \usepackage{verbatim}
    \usepackage{algorithm}
    \usepackage{color}
    \usepackage{empheq}
    \usepackage{booktabs}
    \usepackage{rotating}

 \usepackage{varwidth}
\usepackage{bm}
\usepackage{array}
\usepackage{makecell}
\usepackage{multirow}
\usepackage{pdflscape}
\usepackage{subfigure}
\usepackage{boxedminipage}
\usepackage{multirow}
\usepackage{cases}

\allowdisplaybreaks

\begin{document}
\title{Finite difference method for nonlinear damped viscoelastic Euler-Bernoulli beam model
}


\author{Wenlin Qiu \and Xiangcheng Zheng \and Tao Guo \and Xu Xiao
}


\institute{W. Qiu \at
             School of Mathematics, Shandong University, Jinan, Shandong 250100, P. R. China \\
              \email{wlqiu@sdu.edu.cn}           
           \and
           X. Zheng (Corresponding author) \at
             School of Mathematics, State Key Laboratory of Cryptography and Digital Economy Security, Shandong University, Jinan, 250100, P. R. China \\
              \email{xzheng@sdu.edu.cn}
           \and
           T. Guo \at
           School of Mathematics and Statistics, Hunan Normal University, Changsha, Hunan 410081, P. R. China \\
            \email{guotao6613@163.com}
           \and
           X. Xiao \at
              School of Mathematics and Statistics, Guangxi Normal University, Guilin,  Guangxi 541004, P. R. China \\
              \email{xiaoxu961004@gmail.com}
}

\date{Received: date / Accepted: date}

\maketitle

\begin{abstract}
We propose and analyze the numerical approximation for a viscoelastic Euler-Bernoulli beam model containing a nonlinear strong damping coefficient. The finite difference method is used for spatial discretization, while the backward Euler method and the averaged PI rule are applied for temporal discretization. The long-time stability and the finite-time error estimate of the numerical solutions are derived for both the semi-discrete-in-space scheme and the fully-discrete scheme. Furthermore, the Leray-Schauder theorem is used to derive the existence and uniqueness of the fully-discrete numerical solutions. Finally, the numerical results verify the theoretical analysis.

\keywords{viscoelastic Euler-Bernoulli beam, nonlinear damping, long-time stability, error estimate}

\subclass{35L75 \and 65M15 \and 65M22 \and 45K05 \and 45E10}
\end{abstract}

\section{Introduction}
This article proposes a finite difference method for solving a nonlinear damped viscoelastic Euler-Bernoulli beam model \cite{Cavalcanti1}
   \begin{equation}\label{eq1.1}
   \begin{split}
        u_{tt}(x, t) + q(t)u_t(x, t) + u_{xxxx}(x, t) - \int_0^{t}\beta(t-s)u_{xxxx}(x,s)ds = f(x, t),
   \end{split}
   \end{equation}
for $(x,t)\in (0,1) \times (0,T]$ with $T>0$ being either finite or infinite that will be specified in different cases, including a nonlinear strong damping coefficient (cf. \cite{Emm} and \cite[Section 6]{Cannarsa})
\begin{equation}\label{eq1.2}
     \begin{split}
          q(t) = G\left(  \int_{0}^{1}|u_{xx}(x,t)|^2dx \right), \quad t\geq 0,
     \end{split}
     \end{equation}
subject to initial values
     \begin{equation}\label{eq1.3}
     \begin{split}
          u(x,0)=u_0(x),  \quad  u_t(x,0)=u_1(x),  \quad  x\in (0,1),
     \end{split}
     \end{equation}
and hinged boundary conditions
     \begin{equation}\label{eq1.4}
     \begin{split}
         u(0,t)= u(1,t) =  u_{xx}(0,t)= u_{xx}(1,t) =0,  \quad  t\in [0,T],
     \end{split}
     \end{equation}
 where $G(v)$ in \eqref{eq1.2} is a function between $\mathbb{R}^+$, $u_0$, $u_1$ and $f$ are given functions. Furthermore, $\beta(t)\in L^1(0,\infty)$ is either an oscillatory kernel  \cite{Bru,Cannarsa1,Cannarsa}
     \begin{equation}\label{eq1.5}
     \begin{split}
          \beta(t) = \frac{e^{-\sigma t}t^{\alpha-1}\cos(\gamma t)}{\Gamma(\alpha)}, \quad \sigma> 1, \quad 0\leq \gamma \leq \sigma, \quad  \alpha = \frac{1}{2}, 1,
     \end{split}
     \end{equation}
where the power function $t^{\alpha-1}$ describes the viscoelastic behavior \cite{Lopez}, the exponential factor $e^{-\sigma t}$ has a tempering effect on the power law \cite{Sabzikar} and the trigonometric function $\cos(\gamma t)$ describes the oscillatory feature in time \cite{ZH}, or a non-oscillatory kernel that is commonly used in tempered fractional or nonlocal problems \cite{DenLiTia,LiDenZha,Sabzikar}
\begin{equation}\label{eq1.6}
     \begin{split}
          \beta(t) = \frac{e^{-\sigma t}t^{\alpha-1}}{\Gamma(\alpha)}, \quad \sigma> 1, \quad 0< \alpha \leq 1.
     \end{split}
     \end{equation}
Here $\Gamma(\cdot)$ denotes the standard Gamma function.

Viscoelastic Euler-Bernoulli beam equation is used to describe the mechanical behavior of beams with viscoelastic effects under bending and vibration conditions, with extensive applications in several circumstances \cite{Findley,Timoshenko}.
 For instance,
 viscoelastic materials are widely used in shock-absorbing devices in building structures, which can effectively absorb the vibration energy caused by earthquakes. The Euler-Bernoulli type viscoelastic beam equation is used to analyze the dynamic response of these materials under the action of seismic forces, helping to optimize the design of buildings and improve their earthquake resistance \cite{Christensen,Timoshenko}. 

For the nonlinear damped Euler-Bernoulli equation in the form like \eqref{eq1.1}, some scholars have conducted relevant research on the theoretical analysis.
Cavalcanti et al. \cite{Cavalcanti1} discussed and proved the existence of global solutions for \eqref{eq1.1} and decay rates of the energy. Jorge and Ma \cite{Jorge} derived the well-posedness and asymptotic stability of the solutions with perturbations of $p$-Laplacian type. Yang \cite{Yang} derived the exponential decay results of the energy based on an appropriate Lyapunov function. Conti and Geredeli \cite{Conti} proved the existence of smooth global attractors. Araujo et al. \cite{Araujo} considered the variational inequality for the plate equation with a terminal memory term of $p$-Laplacian.

Despite significant progresses on theoretical investigations, the corresponding numerical studies for hyperbolic integro-differential equations with nonlinear coefficients are far from well developed. There exist substantial numerical analysis works on linear versions of (\ref{eq1.1}) \cite{Allegretto,Cannon,Fairweather,Karaa,Larsson,Pani,Saedpanah,Xu2,XuD4}. For nonlinear problems, Yanik and Fairweather \cite{Yanik} proposed a discrete-time collocation approximation for a hyperbolic integro-differential equation with a nonlinear diffusivity coefficient depending on $u$. Tan et al. \cite{Tan} recently analyzed a fully-discrete two-grid finite element method for a hyperbolic integro-differential equation with a nonlinear coefficient depending on $u$. Qiu et al. \cite{QiuZhe} considered a fully-discrete finite element method for a hyperbolic integro-differential equation with the viscous nonlinear-nonlocal damping. For model (\ref{eq1.1}), the numerical analysis remains untreated due to the difficulties caused by, e.g.,  the nonlinear strong damping coefficient.

Motivated by aforementioned discussions, we conduct numerical approximation and analysis for model (\ref{eq1.1}). We use the finite difference method to obtain a semi-discrete-in-space scheme and derive the long-time stability based on a transformation of the kernel, and then prove its convergence by using technical splittings (e.g., the equation \eqref{eq2.34}) to accommodate the difficulties caused by the nonlinear strong damping. We then apply the backward Euler method and averaged PI rule to establish the fully-discrete difference scheme. By energy argument, we prove the long-time stability of the fully-discrete scheme with the help of some transformations of summation terms (e.g., the equation \eqref{eq3.12}), following which we give the convergence analysis. Finally, we apply the Leray-Schauder theorem to derive the existence and uniqueness of fully-discrete numerical solutions.

The rest of the paper is organized as follows. In Section \ref{sec2}, we formulate and analyze a spatial semi-discrete difference scheme. Then we further establish the fully-discrete difference scheme and deduce some theoretical results in Sections \ref{sec3} and \ref{sec4}, respectively. Section \ref{sec5} is devoted to validating the theoretical analysis by several numerical examples. Finally, some concluding remarks are presented in Section \ref{sec6}.

\section{Spatial semi-discrete scheme}\label{sec2}

In this section, we shall formulate and analyze a spatial semi-discrete difference scheme for problem \eqref{eq1.1}-\eqref{eq1.4}.

\subsection{Construction of spatial semi-discrete scheme}
First, we introduce some notations for further analysis. Given a mesh $x_j=jh$, $j=0,1,\cdots, J$ with the spatial step $h=1/J$, and $J$ is a positive integer. In the subsequent analysis, we shall denote that $U_0(t)=U_J(t)=0$ for $t\in (0,T]$. Next, introduce some difference-quotient notations as follows
\begin{equation*}
     \begin{split}
        & u_j(t) = u(x_j, t), \quad U_j(t) \simeq u(x_j, t), \quad
        (U_j(t))_{x} = \frac{U_{j+1}(t)-U_j(t)}{h}, \\
        & (U_j(t))_{\Bar{x}} = \frac{U_{j}(t)-U_{j-1}(t)}{h}, \quad  (U_j(t))_{x\Bar{x}} = \frac{(U_j(t))_{x} - (U_j(t))_{\Bar{x}}}{h}, \\
        & (U_j(t))_{xx\Bar{x}\Bar{x}}
        = \frac{1}{h^4}\left[ U_{j+2}(t) - 4U_{j+1}(t)  + 6U_{j}(t)  - 4U_{j-1}(t) + U_{j-2}(t)   \right].
     \end{split}
     \end{equation*}
Let the notations $\Vec{V}=(V_1, V_2, \cdots, V_{J-1})^{\top}$ and $\Vec{W}=(W_1, W_2, \cdots, W_{J-1})^{\top}$ be the real vectors. Then, denote the following discrete $L^2$ inner product, discrete $L^2$ norm and discrete $L^{\infty}$ norm
\begin{equation}\label{eq2.2}
     \begin{split}
         \langle \Vec{V}, \Vec{W} \rangle
         = h \sum\limits_{j=1}^{J-1} V_j W_j, \quad \|\Vec{W}\| = \sqrt{\langle \Vec{W}, \Vec{W} \rangle}, \quad \|\Vec{W}\|_{\infty} = \max\limits_{1\leq j \leq J-1}\left| W_j\right|.
     \end{split}
     \end{equation}
Then, the following relation holds (see \cite[(2.9)]{Lopez})
\begin{equation*}
     \begin{split}
         \langle \Vec{V}, (\Vec{W})_{x\Bar{x}} \rangle
         = - h \sum\limits_{j=1}^{J-1} (V_j)_x (W_j)_x,
     \end{split}
     \end{equation*}
and we further have (see \cite[Lemma 3.2]{Hu})
\begin{equation}\label{eq2.4}
     \begin{split}
         \langle \Vec{W}, (\Vec{W})_{xx\Bar{x}\Bar{x}} \rangle
         = \langle (\Vec{W})_{x\Bar{x}}, (\Vec{W})_{x\Bar{x}} \rangle,
     \end{split}
     \end{equation}
for $W_{-1}= -W_{1}$ and $W_{J+1}=-W_{J-1}$.

Based on the above difference operators and some notations, we replace $u_j(t)$ with its numerical approximation $U_j(t)$ and obtain the following spatial semi-discrete difference scheme with $f_j(t)=f(x_j,t)$, that is
\begin{equation}\label{eq2.5}
   \begin{split}
        U_j''(t) + G\left( \|\Vec{U}_{x\Bar{x}}(t)\|^2 \right) U_j'(t) + (U_{j})_{xx\Bar{x}\Bar{x}}(t) - \int_0^{t}\beta(t-s)(U_{j})_{xx\Bar{x}\Bar{x}}(s)ds = f_j(t),
   \end{split}
   \end{equation}
\begin{equation}\label{eq2.6}
   \begin{split}
        & U_0(t) = U_{J}(t) = 0, \quad 0< t \leq T, \\
        & U_{-1}(t) = -U_{1}(t), \quad U_{J+1}(t) = -U_{J-1}(t), \quad 0< t \leq T,
   \end{split}
   \end{equation}
\begin{equation}\label{eq2.7}
   \begin{split}
        U_j(0) = u_0(x_j), \quad U_j'(0) = u_1(x_j), \quad j = 1,2,\cdots, J-1.
   \end{split}
   \end{equation}

\subsection{Long-time stability}

We first refer the following lemma to support subsequent analysis, cf. \cite{Cannarsa1} and \cite[Lemma 1.1]{Zhao}.
\begin{lemma}\label{lemma1.1} Let $\beta(t)$ be given in \eqref{eq1.5} or \eqref{eq1.6}. Then, the kernel $K(t)=\int_t^{\infty}\beta(s)ds$ is of positive type, such that $K(\infty)=0$ and $K(0):=K_0<1$.
 \end{lemma}

Then we make the following assumptions on $G(v)$ to perform analysis:

($\mathbf{S1}$) There exist positive constants $g_0$ and $g_1$ such that for $0\leq v \leq C$, $G(v)\leq g_1$, and for $v \geq 0$, $G(v)\geq g_0$; \par
($\mathbf{S2}$) $G(v)$ is continuously differentiable function with $0\leq G'(v)\leq L$ for $v\geq 0$, where $L$ is the Lipschitz constant.

\begin{remark}
 The ($\mathbf{S2}$) implies that $G(v)$ is Lipschitz continuous, i.e., for any $v_1,v_2> 0$, it holds that $|G(v_1)-G(v_2)|\leq L|v_1-v_2|$. For instance, $G(v)=1+v$ or $G(v)=\sqrt{1+v}$ satisfies the assumptions ($\mathbf{S1}$)-($\mathbf{S2}$).
\end{remark}

\vskip 1mm
Below, we shall establish the long-time stability for the semi-discrete difference scheme \eqref{eq2.5}-\eqref{eq2.7}.

Denote the notation $\Vec{U}(t)=[U_1(t), U_2(t), \cdots, U_{J-1}(t)]^{\top}$, and assume that $f(\cdot,t)\in L^1(0,\infty)$, that is
\begin{equation}\label{eq2.8}
   \begin{split}
        \|\Vec{f}\|_1 : = \int_0^{\infty} \|\Vec{f}(t)\|dt < \infty, \quad \Vec{f}(t)=[f_1(t), f_2(t), \cdots, f_{J-1}(t)]^{\top},
   \end{split}
   \end{equation}
then the following long-time stability ($T\rightarrow \infty$) holds.

\begin{theorem} \label{theorem2.1}
Let the assumption ($\mathbf{S1}$) hold. Suppose that $\beta(t)$ is denoted by \eqref{eq1.5} or \eqref{eq1.6}, $\Vec{f}(t)$ satisfies \eqref{eq2.8} and let $\Vec{U}(t)$ be the solution of the semi-discrete scheme \eqref{eq2.5}-\eqref{eq2.7}. If $u_0(x)\in C^4([0,1])$, then for any $0<t\leq T\leq \infty$, $\mu_0=1-K_0$ and $C_0=\max\limits_{0\leq t\leq \infty} K(t)$, it holds that
\begin{equation*}
 \begin{split}
     \frac{1}{2}\|\Vec{U}'(t)\|^2
       & + g_0\int_0^{t} \|\Vec{U}'(t)\|^2dt
      +  \frac{\mu_0}{4}\|\Vec{U}_{x\Bar{x}}(t)\|^2 \\
      & \leq e^{\mu_0/8}\left[  \|\Vec{U}'(0)\|^2 + \left(1+2C_0+\frac{2C_0^2}{\mu_0}\right)\|\Vec{U}_{x\Bar{x}}(0)\|^2 + 2\|\Vec{f}\|_1^2\right],
 \end{split}
\end{equation*}
where $\Vec{U}(0)=[u_0(x_1), \cdots, u_0(x_{J-1})]^{\top}$ and $\Vec{U}'(0)=[u_1(x_1), \cdots, u_1(x_{J-1})]^{\top}$.
\end{theorem}
\begin{proof} By taking the inner product of \eqref{eq2.5} with $\Vec{U}'(t)$, and using the fact that $K_t(t)=-\beta(t)$, then we have
\begin{equation*}
   \begin{split}
        \langle \Vec{U}''(t), \Vec{U}'(t) \rangle & +  G\left( \|\Vec{U}_{x\Bar{x}}(t)\|^2 \right) \langle \Vec{U}'(t), \Vec{U}'(t) \rangle + \langle \Vec{U}_{xx\Bar{x}\Bar{x}}(t), \Vec{U}'(t) \rangle \\
        & + \int_0^{t}K_t(t-s) \langle\Vec{U}_{xx\Bar{x}\Bar{x}}(s),\Vec{U}'(t) \rangle ds = \langle \Vec{f}(t), \Vec{U}'(t) \rangle.
   \end{split}
   \end{equation*}
Utilizing the assumption $(\mathbf{S1})$ and \eqref{eq2.4}, we obtain
\begin{equation}\label{eq2.9}
   \begin{split}
      \frac{1}{2}\frac{d}{dt}\|\Vec{U}'(t)\|^2
      & + g_0 \|\Vec{U}'(t)\|^2
      +  \frac{1}{2}\frac{d}{dt}\|\Vec{U}_{x\Bar{x}}(t)\|^2 \\
        & + \int_0^{t}K_t(t-s) \langle\Vec{U}_{x\Bar{x}}(s),\Vec{U}_{x\Bar{x}}'(t) \rangle ds \leq  \| \Vec{f}(t)\| \|\Vec{U}'(t) \|.
   \end{split}
   \end{equation}
In the above formula, applying the integration by parts we see that
\begin{equation}\label{eq2.10}
   \begin{split}
       \int_0^{t}K_t(t-s) \Vec{U}_{x\Bar{x}}(s)ds
       = K(t) \Vec{U}_{x\Bar{x}}(0)
       -  K_0 \Vec{U}_{x\Bar{x}}(t)
       + \int_0^{t}K(t-s) \Vec{U}_{x\Bar{x}}'(s)ds.
   \end{split}
   \end{equation}
Then putting \eqref{eq2.10} into \eqref{eq2.9}, and integrating \eqref{eq2.9} regarding $t$ from 0 to $t_*$,
\begin{equation}\label{eq2.11}
   \begin{split}
       \frac{1}{2}&\|\Vec{U}'(t_{*})\|^2
       + g_0\int_0^{t_*} \|\Vec{U}'(t)\|^2dt
      +  \frac{\mu_0}{2}\|\Vec{U}_{x\Bar{x}}(t_*)\|^2 \\
      &+ \int_0^{t_*} K(t)\langle \Vec{U}'_{x\Bar{x}}(t), \Vec{U}_{x\Bar{x}}(0)\rangle dt + \int_0^{t_*}\int_0^{t} K(t-s)\langle \Vec{U}'_{x\Bar{x}}(s), \Vec{U}'_{x\Bar{x}}(t)\rangle ds dt \\
      & \leq \frac{1}{2} \|\Vec{U}'(0)\|^2 + \frac{1}{2}\|\Vec{U}_{x\Bar{x}}(0)\|^2   + \int_0^{t_*}\| \Vec{f}(t)\| \|\Vec{U}'(t) \|dt,
   \end{split}
   \end{equation}
where $\mu_0=1-K_0>0$, then using Lemma \ref{lemma1.1} to yield that
\begin{equation}\label{eq2.12}
   \begin{split}
       \int_0^{t_*}\int_0^{t} K(t-s)\langle \Vec{U}'_{x\Bar{x}}(s), \Vec{U}'_{x\Bar{x}}(t)\rangle ds dt \geq 0,
   \end{split}
   \end{equation}
and we apply the integration by parts again,
\begin{equation}\label{eq2.13}
   \begin{split}
       \int_0^{t_*} K(t)\ \Vec{U}'_{x\Bar{x}}(t) dt
        = K(t_*)\Vec{U}_{x\Bar{x}}(t_*) - K_0 \Vec{U}_{x\Bar{x}}(0)
        + \int_0^{t_*} \beta(t)\Vec{U}_{x\Bar{x}}(t) dt.
   \end{split}
   \end{equation}
Therefore, substituting \eqref{eq2.12}-\eqref{eq2.13} into \eqref{eq2.11}, we have
\begin{equation}\label{eq2.14}
   \begin{split}
       \frac{1}{2}\|\Vec{U}'(t_{*})\|^2
       & + g_0\int_0^{t_*} \|\Vec{U}'(t)\|^2dt
      +  \frac{\mu_0}{2}\|\Vec{U}_{x\Bar{x}}(t_*)\|^2 \\
      & \leq \frac{1}{2} \|\Vec{U}'(0)\|^2 + \frac{1+2K_0}{2}\|\Vec{U}_{x\Bar{x}}(0)\|^2 + K(t_*) \langle \Vec{U}_{x\Bar{x}}(t_*), \Vec{U}_{x\Bar{x}}(0)  \rangle  \\
      & + \int_0^{t_*} \beta(t) \langle \Vec{U}_{x\Bar{x}}(t), \Vec{U}_{x\Bar{x}}(0) \rangle dt + \int_0^{t_*}\| \Vec{f}(t)\| \|\Vec{U}'(t) \|dt.
   \end{split}
   \end{equation}
After that, we utilize Young's inequality and $K(t)\leq C_0$ to obtain
\begin{equation}\label{eq2.15}
   \begin{split}
      &  K(t_*) \| \Vec{U}_{x\Bar{x}}(t_*)\|  \|\Vec{U}_{x\Bar{x}}(0)  \|
        \leq \frac{\mu_0}{4}\| \Vec{U}_{x\Bar{x}}(t_*)\|^2
        + \frac{C_0^2}{\mu_0} \|\Vec{U}_{x\Bar{x}}(0) \|^2,\\
        & \int_0^{t_*}\| \Vec{f}(t)\| \|\Vec{U}'(t) \|dt
        \leq \frac{1}{4}\sup\limits_{0\leq t \leq t_*}\|\Vec{U}'(t) \|^2
        + \left(\int_0^{t_*}\| \Vec{f}(t)\|dt \right)^2,
   \end{split}
   \end{equation}
and that
\begin{equation}\label{eq2.16}
   \begin{split}
         \int_0^{t_*} \beta(t) \| \Vec{U}_{x\Bar{x}}(t)\| \| \Vec{U}_{x\Bar{x}}(0) \| dt
         &\leq  \int_0^{t_*} \frac{\beta(t)}{2} \| \Vec{U}_{x\Bar{x}}(t)\|^2 dt + \int_0^{t_*} \frac{\beta(t)}{2} dt  \| \Vec{U}_{x\Bar{x}}(0) \|^2.
   \end{split}
   \end{equation}
Noting $\int_0^{\infty} \beta(t) dt = K_0\leq C_0$ and \eqref{eq2.8}, and putting \eqref{eq2.15}-\eqref{eq2.16} into \eqref{eq2.14}, we get
\begin{equation}\label{eq2.17}
   \begin{split}
       \frac{1}{2}\|\Vec{U}'(t_{*})\|^2
       & + g_0\int_0^{t_*} \|\Vec{U}'(t)\|^2dt
      +  \frac{\mu_0}{4}\|\Vec{U}_{x\Bar{x}}(t_*)\|^2 \\
      & \leq \frac{1}{2} \|\Vec{U}'(0)\|^2 + \left(\frac{1}{2}+C_0 +\frac{C_0^2}{\mu_0}\right)\|\Vec{U}_{x\Bar{x}}(0)\|^2 \\
      & + \int_0^{t_*} \beta(t) \|\Vec{U}_{x\Bar{x}}(t) \|^2 dt  + \|\Vec{f}\|_1^2 + \frac{1}{4}\sup\limits_{0\leq t \leq t_*}\|\Vec{U}'(t) \|^2.
   \end{split}
   \end{equation}
Choosing an appropriate $\Bar{t}$ such that $\|\Vec{U}'(\Bar{t}) \|^2=\sup\limits_{0\leq t \leq t_*}\|\Vec{U}'(t) \|^2$, thus we conclude from \eqref{eq2.17} that
\begin{equation}\label{eq2.18}
   \begin{split}
       \frac{1}{4}\sup\limits_{0\leq t \leq t_*}\|\Vec{U}'(t) \|^2
       &
      \leq \frac{1}{2} \|\Vec{U}'(0)\|^2 + \left(\frac{1}{2}+C_0+\frac{C_0^2}{\mu_0}\right)\|\Vec{U}_{x\Bar{x}}(0)\|^2 \\
      & + \int_0^{t_*} \beta(t) \|\Vec{U}_{x\Bar{x}}(t) \|^2 dt  + \|\Vec{f}\|_1^2.
   \end{split}
   \end{equation}
By combining \eqref{eq2.17} and \eqref{eq2.18}, we arrive at
\begin{equation*}
   \begin{split}
       \frac{1}{2}\|\Vec{U}'(t_{*})\|^2
       & + g_0\int_0^{t_*} \|\Vec{U}'(t)\|^2dt
      +  \frac{\mu_0}{4}\|\Vec{U}_{x\Bar{x}}(t_*)\|^2 \\
      & \leq  \|\Vec{U}'(0)\|^2 + \left(1+2C_0+\frac{2C_0^2}{\mu_0}\right)\|\Vec{U}_{x\Bar{x}}(0)\|^2 \\
      & + 2\int_0^{t_*} \beta(t) \|\Vec{U}_{x\Bar{x}}(t) \|^2 dt  + 2\|\Vec{f}\|_1^2, \quad \text{where}\; \int_0^{\infty} \beta(t)dt<1.
   \end{split}
   \end{equation*}
Then, for any $t_*>0$, we apply the Gr\"{o}nwall lemma to finish the proof.
\end{proof}

\subsection{Convergence analysis}
Before establishing the convergence analysis, we first introduce the following key formulas, see \cite[Eq. (10)]{Zhang} and \cite[Eq. (8)]{Hu}, i.e.,
\begin{equation}\label{eq2.20}
   \begin{split}
     & u_{x\Bar{x}}(x_j,t) = u_{xx}(x_j,t)   +  \frac{h^2}{6} \int_0^1 \sum\limits_{\ell=\pm 1}\big[u_{xxxx}(x_j+\ell\theta h,t) \big] (1-\theta)^3 d \theta, \quad t\geq 0, \\
        &  u_{xx\Bar{x}\Bar{x}}(x_j,t) = u_{xxxx}(x_j,t)
        + \frac{h^2}{6}u_{xxxxxx}(x_j,t) + O(h^4), \quad t\geq 0.
   \end{split}
   \end{equation}
In what follows, we shall deduce the error estimate of the semi-discrete difference scheme \eqref{eq2.5}-\eqref{eq2.7} for the problem \eqref{eq1.1}-\eqref{eq1.4}. From \eqref{eq1.1}-\eqref{eq1.4}, we can see that
\begin{equation}\label{eq2.21}
   \begin{split}
        u_j''(t) & + G\left( \int_0^1|u_{xx}(x,t)|^2dx \right)u_j'(t) + (u_{j})_{xx\Bar{x}\Bar{x}}(t) \\
        & - \int_0^{t}\beta(t-s)(u_{j})_{xx\Bar{x}\Bar{x}}(s)ds = f_j(t) + R_1(x_j, t) + R_2(x_j, t),
   \end{split}
   \end{equation}
\begin{equation}\label{eq2.22}
   \begin{split}
        & u_0(t) = u_{J}(t) = 0, \quad 0< t \leq T, \\
        & u_{-1}(t) = -u_{1}(t), \quad u_{J+1}(t) = -u_{J-1}(t), \quad 0< t \leq T,
   \end{split}
   \end{equation}
\begin{equation}\label{eq2.23}
   \begin{split}
        u_j(0) = u_0(x_j), \quad u_j'(0) = u_1(x_j), \quad j = 1,2,\cdots, J-1,
   \end{split}
   \end{equation}
in which,
\begin{equation}\label{eq2.24}
   \begin{split}
       R_1(x_j, t) = (u_{j})_{xx\Bar{x}\Bar{x}}(t)
       - u_{xxxx}(x_j, t),  \quad  j = 1,2,\cdots, J-1,
   \end{split}
   \end{equation}
\begin{equation}\label{eq2.25}
   \begin{split}
       R_2(x_j, t) = \int_0^t \beta(t-s)
        \left[ u_{xxxx}(x_j, s) - (u_{j})_{xx\Bar{x}\Bar{x}}(s) \right]ds.
   \end{split}
   \end{equation}
Then, denote $\Vec{\xi}(t)=\Vec{u}(t)-\Vec{U}(t)$, where $\Vec{u}(t)=[u_1(t), u_2(t), \cdots, u_{J-1}(t)]^{\top}$. By subtracting \eqref{eq2.5}-\eqref{eq2.7} from \eqref{eq2.21}-\eqref{eq2.23}, we arrive at the following error equations
\begin{equation}\label{eq2.26}
   \begin{split}
        \xi_j''(t)  & +  G\left( \|\Vec{U}_{x\Bar{x}}(t)\|^2 \right) \xi_j'(t) + (\xi_{j})_{xx\Bar{x}\Bar{x}}(t) - \int_0^{t}\beta(t-s)(\xi_{j})_{xx\Bar{x}\Bar{x}}(s)ds  \\
        &  = \sum\limits_{m=1}^{2} R_m(x_j, t)  + \left[  G\left( \|\Vec{U}_{x\Bar{x}}(t)\|^2 \right) -  G\left(
        \int_0^1|u_{xx}(x,t)|^2 dx \right)  \right]u_j'(t),
    \end{split}
   \end{equation}
\begin{equation}\label{eq2.27}
   \begin{split}
        & \xi_0(t) = \xi_{J}(t) = 0, \quad (\xi_{0}(t))_{x\Bar{x}} = (\xi_{J}(t))_{x\Bar{x}} = 0, \quad 0< t \leq T, \\
        & \xi_j(0) = 0, \quad \xi_j'(0) = 0, \quad j=1,2,\cdots, J-1.
   \end{split}
   \end{equation}
Then, we further arrive at the following theorem.

\begin{theorem} \label{theorem2.2}
Let the assumptions ($\mathbf{S1}$)-($\mathbf{S2}$) hold. Let $\beta(t)$ be denoted by \eqref{eq1.5} or \eqref{eq1.6}, $\Vec{U}(t)$ be the solution of the semi-discrete scheme \eqref{eq2.5}-\eqref{eq2.7}, and $\Vec{u}(t)$ be the solution of \eqref{eq2.21}-\eqref{eq2.23}. If satisfying
\begin{equation*}
 \begin{split}
    |u_{xxxx}(x,t)|\leq C, \quad |u_{xxxxxx}(x,t)|\leq C, \quad (x,t) \in [0,1]\times (0,T],
 \end{split}
\end{equation*}
then for any $0<t\leq T<\infty$, it holds that
\begin{equation*}
 \begin{split}
    \sqrt{\|\Vec{u}'(t)-\Vec{U}'(t)\|^2
       + \int_0^{t} \|\Vec{u}'(t)-\Vec{U}'(t)\|^2dt
      +  \|\Vec{u}_{x\Bar{x}}(t)-\Vec{U}_{x\Bar{x}}(t)\|^2}  \leq C h^2.
 \end{split}
\end{equation*}
\end{theorem}

\begin{proof}
    Taking the inner product of \eqref{eq2.26} with $\Vec{\xi}'(t)$, we obtain that
\begin{equation}\label{eq2.28}
   \begin{split}
      \frac{1}{2}\frac{d}{dt}\|\Vec{\xi}'(t)\|^2
      & + g_0 \|\Vec{\xi}'(t)\|^2
      +  \frac{1}{2}\frac{d}{dt}\|\Vec{\xi}_{x\Bar{x}}(t)\|^2 \\
        & + \int_0^{t}K_t(t-s) \langle\Vec{\xi}_{x\Bar{x}}(s),\Vec{\xi}_{x\Bar{x}}'(t) \rangle ds \leq  \sum\limits_{m=1}^{2} \langle \Vec{R}_m(t), \Vec{\xi}'(t) \rangle \\
        & + \left[  G\left( \|\Vec{U}_{x\Bar{x}}(t)\|^2 \right) -  G\left(
        \int_0^1|u_{xx}(x,t)|^2 dx \right)  \right] \langle \Vec{u}'(t), \Vec{\xi}'(t) \rangle.
   \end{split}
   \end{equation}
From \eqref{eq2.10} with $\Vec{\xi}_{x\Bar{x}}(0)=0$, we similarly get
\begin{equation*}
   \begin{split}
       \int_0^{t}K_t(t-s) \Vec{\xi}_{x\Bar{x}}(s)ds
       =
       -  K_0 \Vec{\xi}_{x\Bar{x}}(t)
       + \int_0^{t}K(t-s) \Vec{\xi}_{x\Bar{x}}'(s)ds,
   \end{split}
   \end{equation*}
which implies that
\begin{equation}\label{eq2.29}
   \begin{split}
        \int_0^{t}K_t(t-s) \langle\Vec{\xi}_{x\Bar{x}}(s),\Vec{\xi}_{x\Bar{x}}'(t) \rangle ds
        = \frac{-K_0}{2}\frac{d}{dt}\|\Vec{\xi}_{x\Bar{x}}(t)\|^2
       + \int_0^{t}K(t-s) \langle \Vec{\xi}_{x\Bar{x}}'(s), \Vec{\xi}_{x\Bar{x}}'(t) \rangle ds.
   \end{split}
   \end{equation}
After that, we analyze the final term of the right-hand side of \eqref{eq2.28}. First, we rewrite
\begin{equation*}
    \begin{split}
          \int_0^1 |u_{xx}(x,t)|^2dx = \sum\limits_{j=1}^{J} \int_{x_{j-1}}^{x_j}|u_{xx}(x,t)|^2dx, \quad t\geq 0.
    \end{split}
\end{equation*}
Then, defining $v=u_{xx}$ and $\varphi(v)=v^2$, we have
\begin{equation*}
    \begin{split}
          \frac{\partial^2 \varphi(v)}{\partial x^2} = \varphi_{vv}(v) \left(\frac{\partial v}{\partial x}\right)^2  + \varphi_v \frac{\partial^2 v}{\partial x^2} = 2(u_{xxx})^2 + 2u_{xx}u_{xxxx}.
    \end{split}
\end{equation*}
Thus, if satisfying $\sum\limits_{m=2}^{4}\left|\frac{\partial^m u}{\partial x^{m}}(x,t) \right|\leq C$, then $\left\| \frac{d^2 \varphi(v)}{dx^2} \right\|_{L^{\infty}} \leq c_0 < \infty$. Following the well-known trapezoid inequality that
\begin{equation*}
    \begin{split}
          \left| \int_{x_{j-1}}^{x_j}|u_{xx}(x,t)|^2dx - h \frac{|u_{xx}(x_{j-1},t)|^2+|u_{xx}(x_j,t)|^2}{2}  \right| \leq \frac{h^3}{12} \left\| \frac{d^2 \varphi(v)}{dx^2} \right\|_{L^{\infty}}
          \leq \frac{c_0}{12}h^3.
    \end{split}
\end{equation*}
By using the above inequality, \eqref{eq2.2} and \eqref{eq1.4}, we obtain that
\begin{equation}\label{eq2.30}
   \begin{split}
        \left| \int_0^1 |u_{xx}(x,t)|^2dx - \|\Vec{u}_{xx}(t)\|^2 \right| \leq \frac{c_0}{12}h^2, \quad t\geq 0.
   \end{split}
   \end{equation}
Thence, based on the assumption ($\mathbf{S2}$) and \eqref{eq2.30}, one yields
\begin{equation}\label{eq2.31}
   \begin{split}
        \left| G\left(\int_0^1 |u_{xx}(x,t)|^2dx \right) - G\left(\|\Vec{u}_{xx}(t)\|^2 \right) \right|  \leq \frac{L c_0}{12}h^2, \quad t\geq 0.
   \end{split}
   \end{equation}
Then, applying the triangle inequality, the assumption ($\mathbf{S2}$) and \eqref{eq2.20} to get
\begin{equation}\label{eq2.32}
   \begin{split}
        \left| G\left(\|\Vec{u}_{xx}(t)\|^2 \right) - G\left(\|\Vec{u}_{x\Bar{x}}(t)\|^2 \right) \right|  & \leq L(\|\Vec{u}_{xx}(t)\|+ \|\Vec{u}_{x\Bar{x}}(t)\|) \|\Vec{u}_{xx}(t)-\Vec{u}_{x\Bar{x}}(t)\| \\
        & \leq  C\|\Vec{u}_{xx}(t)-\Vec{u}_{x\Bar{x}}(t)\| \leq Ch^2.
   \end{split}
   \end{equation}
In addition, similar to the analysis of \eqref{eq2.32}, employing Theorem \ref{theorem2.1} to get
\begin{equation}\label{eq2.33}
   \begin{split}
        \left| G\left(\|\Vec{u}_{x\Bar{x}}(t)\|^2 \right) - G\left(\|\Vec{U}_{x\Bar{x}}(t)\|^2 \right) \right|  & \leq L(\|\Vec{u}_{x\Bar{x}}(t)\|+ \|\Vec{U}_{x\Bar{x}}(t)\|) \|\Vec{\xi}_{x\Bar{x}}(t)\|  \leq  C\|\Vec{\xi}_{x\Bar{x}}(t)\|.
   \end{split}
   \end{equation}
Combining \eqref{eq2.31}-\eqref{eq2.33}, then it holds that
\begin{equation}\label{eq2.34}
   \begin{split}
        \left| G\left( \|\Vec{U}_{x\Bar{x}}(t)\|^2 \right) -  G\left(
        \int_0^1|u_{xx}(x,t)|^2 dx \right)  \right|  & \leq   C\left( h^2+ \|\Vec{\xi}_{x\Bar{x}}(t)\| \right).
   \end{split}
   \end{equation}
Next, substituting \eqref{eq2.29} and \eqref{eq2.34} into \eqref{eq2.28}, using the Cauchy-Schwarz inequality and $\| \Vec{u}'(t)\|\leq C$ (see \cite[(3.8)]{Cavalcanti1}), then we further get
\begin{equation}\label{eq2.35}
   \begin{split}
      \frac{1}{2}\frac{d}{dt}\|\Vec{\xi}'(t)\|^2
      & + g_0 \|\Vec{\xi}'(t)\|^2
      +  \frac{\mu_0}{2}\frac{d}{dt}\|\Vec{\xi}_{x\Bar{x}}(t)\|^2 + \int_0^{t}K(t-s) \langle\Vec{\xi}_{x\Bar{x}}'(s),\Vec{\xi}_{x\Bar{x}}'(t) \rangle ds \\
        &  \leq  \sum\limits_{m=1}^{2} \| \Vec{R}_m(t)\|\| \Vec{\xi}'(t) \| + C\left( h^2+ \|\Vec{\xi}_{x\Bar{x}}(t)\| \right) \| \Vec{\xi}'(t) \| \\
        & \leq C\sum\limits_{m=1}^{2} \| \Vec{R}_m(t)\|^2+ C\left( h^4+ \|\Vec{\xi}_{x\Bar{x}}(t)\|^2 \right)  + \frac{g_0}{2} \| \Vec{\xi}'(t) \|^2.
   \end{split}
   \end{equation}
Integrating \eqref{eq2.35} regarding $t$ from 0 to $t_*$, and noting $\Vec{\xi}'(0)=\Vec{\xi}_{x\Bar{x}}(0)=0$ and \eqref{eq2.12},
\begin{equation}\label{eq2.36}
   \begin{split}
      \frac{1}{2}\|\Vec{\xi}'(t_*)\|^2
      & + \frac{g_0}{2}\int_0^{t_*} \|\Vec{\xi}'(t)\|^2dt
      +  \frac{\mu_0}{2}\|\Vec{\xi}_{x\Bar{x}}(t_*)\|^2 \\
        & \leq C \int_0^{t_*}\left[ \sum\limits_{m=1}^{2} \| \Vec{R}_m(t)\|^2 + h^4\right]dt + C \int_0^{t_*} \|\Vec{\xi}_{x\Bar{x}}(t)\|^2 dt, \quad t_* \leq T.
   \end{split}
   \end{equation}
Then, using \eqref{eq2.20}, \eqref{eq2.24} and \eqref{eq2.25}, it is easy to yield
\begin{equation}\label{eq2.37}
   \begin{split}
      \sum\limits_{m=1}^{2} \| \Vec{R}_m(t)\|^2 \leq \left(C+ \int_0^T\beta(t)dt \right) h^4 \leq C(T) h^4, \quad t\geq 0.
   \end{split}
   \end{equation}
By putting \eqref{eq2.37} into \eqref{eq2.36}, an application of the Gr\"{o}nwall's lemma with $T<\infty$ completes the proof.
\end{proof}

\section{Fully discrete difference scheme}\label{sec3}
In this section,  we shall consider \eqref{eq1.1} with the kernel \eqref{eq1.5} or \eqref{eq1.6}. We will construct a fully discrete difference scheme and deduce some theoretical results.

\subsection{Construction of fully discrete scheme}
Herein, we will formulate a fully discrete difference scheme based on the spatial semi-discrete scheme \eqref{eq2.5}-\eqref{eq2.7}. First, applying $K'(t)=-\beta(t)$ and denoting $\mu_0:=1-K_0$, we get
\begin{equation*}
    -\int_0^t \beta(t-s)\Delta^2 u(s)ds = K(t)\Delta^2 u_0 - K_0 \Delta^2 u(t) + \int_0^t K(t-s)\Delta^2 u_t(s) ds.
\end{equation*}
Thence, we can rewrite \eqref{eq1.1} as
\begin{align}
    & u_{tt} + q(t)u_t  + \mu_0 u_{xxxx} + \int_0^t K(t-s) u_{txxxx}(s) ds = f(t) - K(t) u_{0,xxxx}. \label{ModelA}
\end{align}
Then, we present some helpful notations for further analysis. Let $N\in \mathbb{Z}^+$ and $\Delta t$ be the uniform time step, and $U^n_j$ be the numerical approximation of $u(x_j,t_n)$, respectively, with the node $t_n=n\Delta t$. Furthermore, define the following difference quotient notations
\begin{equation*}
    \delta_tU^n_j=(U^n_j-U^{n-1}_j)/\Delta t, \quad n\geq 1, \quad  \delta_t^2 U^n_j=\delta_t(\delta_tU^n_j), \quad n\geq 2.
\end{equation*}
Then, to approximate the integral term in \eqref{ModelA}, we introduce the averaged first-order quadrature rule of product-integration type, namely
\begin{equation}\label{eq3.1}
   \begin{split}
        \Bar{q}_n(\phi) & = \frac{1}{\Delta t} \int_{t_{n-1}}^{t_n}\int_{0}^{t} K(t-s)\phi(t_p)ds dt = \sum\limits_{p=1}^{n}w_{np}\phi(t_p) \\
        & \simeq \sum\limits_{p=1}^{n}\int_{t_{p-1}}^{t_p}K(t_n-s)\phi(s)ds,
   \end{split}
   \end{equation}
where the weights
\begin{equation*}
   \begin{split}
        w_{np} = \frac{1}{\Delta t} \int_{t_{n-1}}^{t_n}\int_{t_{p-1}}^{\min(t,t_p)} K(t-s) dsdt>0, \quad 1\leq p \leq n.
   \end{split}
   \end{equation*}
Based on the backward Euler method and \eqref{eq2.5}-\eqref{eq2.7}, we replace $u_j(t_n)$ with its numerical approximation $U_j^n$ and yield the following fully discrete difference scheme with $f_j^n=f(x_j,t_n)$, i.e.,
\begin{equation}\label{eq3.2}
   \begin{split}
        \delta_t^2 U_j^n + G\left( \|\Vec{U}_{x\Bar{x}}^n\|^2 \right) \delta_t U_j^n & + \mu_0 (U_{j}^n)_{xx\Bar{x}\Bar{x}}+  \sum\limits_{p=1}^{n}w_{np} (\delta_tU_{j}^p)_{xx\Bar{x}\Bar{x}} \\
        & = f_j^n - K(t_n)U^0_{xx\Bar{x}\Bar{x}}, \quad n\geq 2,
   \end{split}
   \end{equation}
\begin{equation}\label{eq3.3}
   \begin{split}
        & U_0^n = U_{J}^n = 0, \quad U_{-1}^n = -U_{1}^n, \quad U_{J+1}^n = -U_{J-1}^n, \quad  1\leq n \leq N,
   \end{split}
   \end{equation}
\begin{equation}\label{eq3.4}
   \begin{split}
        U_j^0 = u_0(x_j), \quad \delta_tU_j^1 = u_1(x_j), \quad j = 1,2,\cdots, J-1.
   \end{split}
   \end{equation}
Here $N=T/\Delta t$ for finite or infinite $T$ that will be specified in different cases. From \eqref{eq3.4}, we obtain the relation as follows
\begin{equation}\label{eq3.5}
   \begin{split}
       U_j^1 = U_j^0 + \Delta t u_1(x_j)  =  u_0(x_j) + \Delta t u_1(x_j), \quad j = 1,2,\cdots, J-1.
   \end{split}
   \end{equation}

\subsection{Existence of numerical solutions}

In this subsection, we shall first derive the existence of the solution for \eqref{eq3.2}-\eqref{eq3.4}. To further derive the existence, we give the following Leray-Schauder theorem \cite[Theorem 6.3.3, pp.~162--163]{Ortega}.
\begin{theorem} \label{thm.LS}
    Assume that $\mathbf{C}$ is an open, bounded set in $\mathbb{R}^{m}$ containing the origin and $\Pi$: $\mathbf{C} \subset \mathbb{R}^{m} \rightarrow \mathbb{R}^{m}$ is a continuous mapping. If $\Pi y \neq \hat{\lambda} y$ whenever $\hat{\lambda}>1$ and $y$ belongs to the boundary of $\mathbf{C}$, then $\Pi$ has a fixed point in the closure of $\mathbf{C}$.
\end{theorem}

\vskip 1mm
Below we apply this theorem to illustrate that, without any restrictions regarding time-space step sizes $\Delta t$, $h$ and initial-value conditions $U^0$, $U^1$.

Denote $\Vec{V}^n = \delta_t\Vec{U}^n$ for $1\leq n \leq N$. Here we rewrite \eqref{eq3.2} as
\begin{equation}\label{eq3.39}
   \begin{split}
        \delta_t \Vec{V}^n & + G\left( \|\Delta t \Vec{V}_{x\Bar{x}}^n + Y_{x\Bar{x}}\|^2 \right) \Vec{V}^n +  [w_{nn}+\mu_0\Delta t]\Vec{V}^n_{xx\Bar{x}\Bar{x}}
       \\
       & = \Vec{f}^n- \mu_0 Y_{xx\Bar{x}\Bar{x}} - \sum\limits_{p=1}^{n-1}w_{np} \Vec{V}^p_{xx\Bar{x}\Bar{x}} - K(t_n)\Vec{U}^0_{xx\Bar{x}\Bar{x}}
       := \Vec{F}^n,
   \end{split}
   \end{equation}
where $Y:=\Vec{U}^{n-1}$ and $\Vec{F}^n$ are known at each time level, provided initial-value conditions. Then we present the following theorem.

\begin{theorem}\label{theorem3.3}
Based on the conditions in Theorem \ref{theorem2.1}, given $J, \Delta t >0$, $T\leq \infty$ and $\Vec{U}^0,\Vec{U}^1\in \mathbb{R}^{J-1}$, \eqref{eq3.2} has a solution $\Vec{U}^2, \Vec{U}^3, \cdots, \Vec{U}^N$.
\end{theorem}

\begin{proof} We shall demonstrate that, given $\Vec{U}^0, \Vec{U}^1, \cdots, \Vec{U}^{n-1}$, equation \eqref{eq3.2} for $\Vec{U}^n$ has a solution. Since $\Vec{U}^n=Y+\Delta t \Vec{V}^n$, we only need to prove that, given $\Vec{V}^1, \Vec{V}^2, \cdots, \Vec{V}^{n-1}$, \eqref{eq3.2} for $\Vec{V}^n$ has a solution.

If we denote a mapping $\Pi: \mathbb{R}^{J-1}\rightarrow\mathbb{R}^{J-1}$ via
\begin{equation*}
    \begin{split}
        \Pi(W) = -\Delta t \left\{ G\left( \|\Delta t W_{x\Bar{x}} + Y_{x\Bar{x}}\|^2 \right) W  + [w_{nn}+\mu_0\Delta t] W_{xx\Bar{x}\Bar{x}} \right\},
    \end{split}
\end{equation*}
then $\Vec{V}^n$ is a solution of \eqref{eq3.2} iff
\begin{equation*}
    \begin{split}
       \Vec{V}^n = \Pi(\Vec{V}^n) + Z, \quad \text{where} \; Z = \Vec{V}^{n-1} + \Delta t \Vec{F}^n.
    \end{split}
\end{equation*}
Hence, we have to illustrate that, the mapping $\Xi(\cdot)=\Pi(\cdot)+Z$ has a fixed point. We next shall utilize the Leray-Schauder theorem, see Theorem \ref{thm.LS}. Considering an open ball $\mathbf{C}=\mathcal{B}(0,r)$ in $\mathbb{R}^{J-1}$ with the norm $\|\cdot\|$ in \eqref{eq2.2}. Suppose that for $W$ in the boundary of $\mathbf{C}$ and $\hat{\lambda}>1$,
\begin{equation}\label{eq3.40}
   \begin{split}
        \hat{\lambda}W =\Xi(W)=\Pi(W)+Z.
   \end{split}
   \end{equation}
By taking the inner product of \eqref{eq3.40} with $W$, and using the assumption ($\mathbf{S1}$) and \eqref{eq2.4}, we have
\begin{equation*}
    \begin{split}
       \hat{\lambda}\|W\|^2 \leq \langle Z, W \rangle \Rightarrow \hat{\lambda}\|W\|^2 \leq \frac{1}{4}\|W\|^2 + \|Z\|^2,
    \end{split}
\end{equation*}
which implies that
\begin{equation*}
    \begin{split}
       \hat{\lambda}  \leq \frac{1}{4} + \frac{1}{r^2}\|Z\|^2.
    \end{split}
\end{equation*}
For a large $r$, the above formula contradicts the hypothesis $\hat{\lambda} >1$. Thence, \eqref{eq3.40} has no solution on the boundary of $\mathbf{C}$, and then Theorem \ref{thm.LS} ensures the existence of a fixed point of $\Xi$ in the closure of $\mathbf{C}$. We then complete the proof.
\end{proof}

\subsection{Long-time stability}

Herein, we shall establish the long-time stability of the fully discrete finite difference scheme \eqref{eq3.2}-\eqref{eq3.4}. First, we will give the following results of long-time stability.

\begin{theorem}\label{theorem3.1}Let the assumption ($\mathbf{S1}$) hold. Let $\beta(t)$ be denoted by \eqref{eq1.5} or \eqref{eq1.6}, $\Vec{f}(t)$ satisfies \eqref{eq2.8} and $\Vec{U}^n=[U_1^n, U_2^n, \cdots, U_{J-1}^n]^{\top}$ be the solution of \eqref{eq3.2}-\eqref{eq3.4}. If $u_0(x), u_1(x) \in C^4([0,1])$, then for $T\leq \infty$, it holds that
     \begin{equation*}
        \begin{split}
        & \frac{1}{2} \|\delta_t \Vec{U}^N\|^2 +
        g_0 \Delta t\sum\limits_{n=2}^N \|\delta_t\Vec{U}^n\|^2 + \frac{\mu_0}{4} \| \Vec{U}^N_{x\Bar{x}}\|^2  \\
       &\leq C \left[ \| \Vec{U}'(0)\|^2 +  \| \Vec{U}_{x\Bar{x}}(0)\|^2
       + (\Delta t)^2\|\Vec{U}_{x\Bar{x}}'(0)\|^2 +  \|\Vec{f}\|_1^2 \right],
        \end{split}
     \end{equation*}
where $\Vec{U}(0)=[u_0(x_1), \cdots, u_0(x_{J-1})]^{\top}$ and $\Vec{U}'(0)=[u_1(x_1), \cdots, u_1(x_{J-1})]^{\top}$.
\end{theorem}
\begin{proof} We take the inner product of \eqref{eq3.2} with $\delta_t \Vec{U}^n$, then for $n \geq 2$,
\begin{equation}\label{eq3.6}
   \begin{split}
        \langle \delta_t^2\Vec{U}^n, \delta_t\Vec{U}^n \rangle & +  G\left( \|\Vec{U}_{x\Bar{x}}^n\|^2 \right) \|\delta_t\Vec{U}^n\|^2 + \mu_0\langle \Vec{U}_{x\Bar{x}}^n, \delta_t\Vec{U}_{x\Bar{x}}^n \rangle \\
        & + \sum\limits_{p=1}^{n}w_{np} \langle \delta_t\Vec{U}_{x\Bar{x}}^p, \delta_t\Vec{U}_{x\Bar{x}}^n \rangle  = \langle \Vec{f}^n, \delta_t\Vec{U}^n \rangle - K(t_n)\langle \Vec{U}_{x\Bar{x}}^0, \delta_t\Vec{U}_{x\Bar{x}}^n \rangle.
   \end{split}
   \end{equation}
First, noting that
\begin{equation}\label{eq3.7}
   \begin{split}
        & \langle \delta_t^2 \Vec{U}^n, \delta_t \Vec{U}^n \rangle
        = \frac{1}{2}\delta_t \|\delta_t \Vec{U}^n\|^2 +
        \frac{1}{2}\Delta t \|\delta_t^2 \Vec{U}^n\|^2, \\
        & \langle  \Vec{U}^n_{x\Bar{x}}, \delta_t \Vec{U}^n_{x\Bar{x}} \rangle
        = \frac{1}{2}\delta_t \| \Vec{U}^n_{x\Bar{x}}\|^2 +
        \frac{1}{2}\Delta t \|\delta_t \Vec{U}^n_{x\Bar{x}}\|^2.
   \end{split}
   \end{equation}
By substituting \eqref{eq3.7} into \eqref{eq3.6}, then we have
\begin{equation}\label{eq3.9}
   \begin{split}
        \frac{1}{2}\delta_t \|\delta_t \Vec{U}^n\|^2 &+
        \frac{1}{2}\Delta t \|\delta_t^2 \Vec{U}^n\|^2  +  G\left( \|\Vec{U}_{x\Bar{x}}^n\|^2 \right) \|\delta_t\Vec{U}^n\|^2 + \frac{\mu_0}{2}\delta_t \| \Vec{U}^n_{x\Bar{x}}\|^2  \\
        &+
        \frac{\mu_0}{2}\Delta t \|\delta_t \Vec{U}^n_{x\Bar{x}}\|^2  + \sum\limits_{p=1}^{n}w_{np} \langle \delta_t\Vec{U}_{x\Bar{x}}^p, \delta_t\Vec{U}_{x\Bar{x}}^n \rangle \\
       & + K(t_n)\langle \Vec{U}_{x\Bar{x}}^0, \delta_t\Vec{U}_{x\Bar{x}}^n \rangle = \langle \Vec{f}^n, \delta_t\Vec{U}^n \rangle.
   \end{split}
   \end{equation}
With the assupmtion ($\mathbf{S1}$), summing for \eqref{eq3.9} regarding $n$ from 2 to $N$, and multiplying $\Delta t$, then we get
\begin{equation}\label{eq3.10}
   \begin{split}
       & \frac{1}{2} \|\delta_t \Vec{U}^N\|^2 +
        g_0 \Delta t\sum\limits_{n=2}^N \|\delta_t\Vec{U}^n\|^2 + \frac{\mu_0}{2} \| \Vec{U}^N_{x\Bar{x}}\|^2  \\
        &+ \Delta t \sum\limits_{n=1}^{N}\sum\limits_{p=1}^{n}w_{np} \langle \delta_t\Vec{U}_{x\Bar{x}}^p, \delta_t\Vec{U}_{x\Bar{x}}^n \rangle + \Delta t\sum\limits_{n=2}^N K(t_n)\langle  \Vec{U}_{x\Bar{x}}^0, \delta_t\Vec{U}_{x\Bar{x}}^n \rangle \\
       &\leq  \frac{1}{2} \|\delta_t \Vec{U}^1\|^2 + \frac{\mu_0}{2} \| \Vec{U}^1_{x\Bar{x}}\|^2
       + (\Delta t) w_{11}\|\delta_t\Vec{U}^1_{x\Bar{x}}\|^2  +\Delta t\sum\limits_{n=2}^N \| \Vec{f}^n\|\| \delta_t\Vec{U}^n \|.
   \end{split}
   \end{equation}
Next, we shall further estimate \eqref{eq3.10}. First, we analyze the fourth term of the left-hand side of \eqref{eq3.10}. Based on \eqref{eq1.6}, Lemma \ref{lemma1.1} shows that the kernel $K(t)$ is of positive type, thus we see that \cite[pp.~63-64]{McLean1}
   \begin{equation}\label{eq3.11}
   \begin{split}
       \Delta t \sum\limits_{n=1}^{N}\sum\limits_{p=1}^{n}w_{np} \langle \delta_t\Vec{U}_{x\Bar{x}}^p, \delta_t\Vec{U}_{x\Bar{x}}^n \rangle \geq 0.
   \end{split}
   \end{equation}
Then for $N\geq 2$, we rewrite
\begin{equation}\label{eq3.12}
   \begin{split}
        \Delta t\sum\limits_{n=2}^N K(t_n)   \delta_t\Vec{U}_{x\Bar{x}}^n
    & = K(t_N)\Vec{U}^N_{x\Bar{x}} - K(t_2)\Vec{U}^1_{x\Bar{x}}
    + \sum\limits_{n=2}^{N-1}[K(t_n)-K(t_{n+1})] \Vec{U}^n_{x\Bar{x}} \\
    & = K(t_N)\Vec{U}^N_{x\Bar{x}} - K(t_2)\Vec{U}^1_{x\Bar{x}}
    + \sum\limits_{n=2}^{N-1}\left[\int_{t_n}^{t_{n+1}}\beta(t)dt \right] \Vec{U}^n_{x\Bar{x}},
   \end{split}
\end{equation}
which further implies that
\begin{equation*}
   \begin{split}
         \Delta t\sum\limits_{n=2}^N K(t_n)\langle  \Vec{U}_{x\Bar{x}}^0, \delta_t\Vec{U}_{x\Bar{x}}^n \rangle
         & \leq C_0(\|\Vec{U}^N_{x\Bar{x}}\| + \|\Vec{U}^1_{x\Bar{x}}\|) \|\Vec{U}^0_{x\Bar{x}}\|
         + \int_{t_2}^{t_{N}}|\beta(t)|dt \|\Vec{U}^0_{x\Bar{x}}\|^2
         \\
         & + \frac{1}{4}\sum\limits_{n=2}^{N-1}\left[\int_{t_n}^{t_{n+1}}|\beta(t)|dt \right] \|\Vec{U}^n_{x\Bar{x}}\|^2.
   \end{split}
   \end{equation*}
After that, we use the Young's inequality to obtain
\begin{equation*}
   \begin{split}
         K_0(\|\Vec{U}^N_{x\Bar{x}}\| + \|\Vec{U}^1_{x\Bar{x}}\|) \|\Vec{U}^0_{x\Bar{x}}\|
         \leq \frac{\mu_0}{4}\|\Vec{U}^N_{x\Bar{x}}\|^2 + \frac{\mu_0}{4}\|\Vec{U}^1_{x\Bar{x}}\|^2 + \frac{2C_0^2}{\mu_0}\|\Vec{U}^0_{x\Bar{x}}\|^2,
   \end{split}
   \end{equation*}
and that
\begin{equation}\label{eq3.14}
   \begin{split}
        \Delta t \sum\limits_{n=2}^{N}\|\Vec{f}^n\| \|\delta_t\Vec{U}^n\|& \leq  \frac{1}{4}\max\limits_{2\leq n \leq N}\|\delta_t\Vec{U}^n\|^2   +  \left( \Delta t \sum\limits_{n=2}^{N}\|\Vec{f}^n\| \right)^2 \\
      & \leq \frac{1}{4}\max\limits_{2\leq n \leq N}\|\delta_t\Vec{U}^n\|^2   + \left( \int_{0}^{t_N}\|\Vec{f}(t)\|dt \right)^2.
   \end{split}
   \end{equation}
Then, substituting \eqref{eq3.11}-\eqref{eq3.14} into \eqref{eq3.10}, we employ \eqref{eq2.8} and \eqref{eq3.4}-\eqref{eq3.5} to yield
\begin{equation}\label{eq3.15}
   \begin{split}
       & \frac{1}{2} \|\delta_t \Vec{U}^N\|^2 +
        g_0 \Delta t\sum\limits_{n=2}^N \|\delta_t\Vec{U}^n\|^2 + \frac{\mu_0}{4} \| \Vec{U}^N_{x\Bar{x}}\|^2  \\
       &\leq  \frac{1}{2} \| \Vec{U}'(0)\|^2 + \left(1+\frac{2C_0^2}{\mu_0}\right) \| \Vec{U}_{x\Bar{x}}(0)\|^2
       + \left( 2C_0 + \mu_0 \right)  (\Delta t)^2\|\Vec{U}_{x\Bar{x}}'(0)\|^2 \\
       & +  \|\Vec{f}\|_1^2 + \frac{1}{4}\sum\limits_{n=2}^{N-1}\left[\int_{t_n}^{t_{n+1}}|\beta(t)|dt \right] \|\Vec{U}^n_{x\Bar{x}}\|^2 + \frac{1}{4}\max\limits_{2\leq n \leq N}\|\delta_t\Vec{U}^n\|^2,
   \end{split}
   \end{equation}
where $w_{11}\leq C_0 \Delta t$, and then we conclude from \eqref{eq3.15} that
\begin{equation*}
   \begin{split}
        \frac{1}{4}& \max\limits_{2\leq n \leq N}\|\delta_t\Vec{U}^n\|^2
       \leq  \frac{1}{2} \| \Vec{U}'(0)\|^2 + \left(1+\frac{2C_0^2}{\mu_0}\right) \| \Vec{U}_{x\Bar{x}}(0)\|^2 +  \|\Vec{f}\|_1^2
       \\
       &+ \left( 2C_0 + \mu_0 \right)  (\Delta t)^2\|\Vec{U}_{x\Bar{x}}'(0)\|^2   + \frac{1}{4}\sum\limits_{n=2}^{N-1}\left[\int_{t_n}^{t_{n+1}}|\beta(t)|dt \right] \|\Vec{U}^n_{x\Bar{x}}\|^2.
   \end{split}
   \end{equation*}
By utilizing \eqref{eq3.15} and the above estimate, it holds that
\begin{equation*}
   \begin{split}
       & \frac{1}{2} \|\delta_t \Vec{U}^N\|^2 +
        g_0 \Delta t\sum\limits_{n=2}^N \|\delta_t\Vec{U}^n\|^2 + \frac{\mu_0}{4} \| \Vec{U}^N_{x\Bar{x}}\|^2  \\
       &\leq   \| \Vec{U}'(0)\|^2 + \left(2+\frac{4C_0^2}{\mu_0}\right) \| \Vec{U}_{x\Bar{x}}(0)\|^2
       + (4C_0+2\mu_0) (\Delta t)^2\|\Vec{U}_{x\Bar{x}}'(0)\|^2  \\
       & +  2\|\Vec{f}\|_1^2 + \frac{1}{2}\sum\limits_{n=2}^{N-1}\left[\int_{t_n}^{t_{n+1}}|\beta(t)|dt \right] \|\Vec{U}^n_{x\Bar{x}}\|^2, \quad N\geq 2.
   \end{split}
   \end{equation*}
Noting that $\lim\limits_{N\rightarrow \infty}\sum\limits_{n=2}^{N-1}\left[\int_{t_n}^{t_{n+1}}|\beta(t)|dt \right] \leq \int_0^{\infty}|\beta(t)|dt <1$, for the above formula, we use the discrete Gr\"{o}nwall's lemma to finish the proof.
\end{proof}

\section{Convergence and uniqueness}\label{sec4}

In this section, we will deduce the convergence and uniqueness of the fully discrete finite difference scheme.

\subsection{Convergence analysis}
In what follows, we shall construct the convergence analysis of the fully discrete difference scheme \eqref{eq3.2}-\eqref{eq3.4}. Before that, we shall assume the regularity of the solution to the initial-value problem \eqref{eq1.1}-\eqref{eq1.3}, for subsequent error estimations. Since the regularity of the nonlinear problem \eqref{eq1.1}-\eqref{eq1.3} cannot be obtained at present, thus,  following the regularity assumption of solutions for the problem in \cite{Baker}, we give the following regularity assumptions applicable to our nonlinear cases.

\textbf{Assumption A.} For $T<\infty$, suppose the solution of \eqref{eq1.1}-\eqref{eq1.3}
\begin{equation*}
   \begin{split}
       u \in C([0,T];H^3(0,1)) \cap C^2([0,T];H^2(0,1)) \cap  C^3((0,T];H^2(0,1)),
   \end{split}
   \end{equation*}
and there exists constants $c_j>0$, for $j=1,2,3$, such that
\begin{equation*}
   \begin{split}
        & |u_{ttt}(x,t)|\leq c_1 t^{\alpha-1},  \quad  \alpha = \frac{1}{2}, 1,   \quad  |u_{ttxxxx}(x,t)|\leq c_2, \\
        &  |u_{txxxxxx}(x,t)|\leq c_3, \quad  (x,t) \in [0,1]\times (0,T].
   \end{split}
   \end{equation*}
It is mentioned that the assumptions in \textbf{Assumption A} will be used in the subsequent error analysis.

\vskip 1mm
Denote $u_j^n:=u(x_j,t_n)$ for $0\leq j \leq J$ and $0\leq n \leq N$. Subsequently, based on the backward Euler method and the quadrature rule \eqref{eq3.1}, we consider \eqref{ModelA} at the point $t=t_n$, then
\begin{equation}\label{eq3.18}
   \begin{split}
        \delta_t^2 u_j^n &+ G\left( \int_0^1 |u_{xx}(x,t_n)|^2dx \right) \delta_t u_j^n  + \mu_0(u_{j}^n)_{xx\Bar{x}\Bar{x}}+  \sum\limits_{p=1}^{n}w_{np}(\delta_t u_{j}^p)_{xx\Bar{x}\Bar{x}} \\
        & = f_j^n - K(t_n)( u_{j}^0)_{xx\Bar{x}\Bar{x}} + \sum\limits_{m=1}^{5}(R_m)_j^n, \quad n\geq 2,
   \end{split}
   \end{equation}
\begin{equation}\label{eq3.19}
   \begin{split}
        & u_0^n = u_{J}^n = 0, \quad u_{-1}^n = -u_{1}^n, \quad u_{J+1}^n = -u_{J-1}^n, \quad  1\leq n \leq N,
   \end{split}
   \end{equation}
\begin{equation}\label{eq3.20}
   \begin{split}
        u_j^0 = u_0(x_j), \quad (u_t)_j^0 = u_1(x_j), \quad j = 1,2,\cdots, J-1,
   \end{split}
   \end{equation}
in which the truncation errors
\begin{equation}\label{eq3.21}
   \begin{split}
        & (R_1)_j^n = \mu_0 \left[ (u_j^n)_{xx\Bar{x}\Bar{x}} - u_{xxxx}(x_j,t_n) \right], \\
        & (R_2)_j^n = \sum\limits_{p=1}^n w_{np} \left[ (\delta_tu_j^p)_{xx\Bar{x}\Bar{x}} - (\delta_tu_j^p)_{xxxx} \right],  \\
        & (R_3)_j^n = \delta_t^2 u_j^n - u_j''(t_n), \quad  n \geq 2, \\
        & (R_4)_j^n = G\left( \int_0^1 |u_{xx}(x,t_n)|^2dx \right) \left[\delta_t u_j^n - u_j'(t_n) \right], \\
        &  (R_5)_j^n = \int_0^{t_n}K(t_n-s) u_{txxxx}(x_j,s)ds - \sum\limits_{p=1}^{n} \int_{t_{p-1}}^{t_p} K(t_n-s) (\delta_tu_j^p)_{xxxx}ds \\
        & \qquad\quad + \sum\limits_{p=1}^{n} \int_{t_{p-1}}^{t_p} K(t_n-s) (\delta_tu_j^p)_{xxxx}ds
        - \sum\limits_{p=1}^{n}w_{np}(\delta_tu_j^p)_{xxxx}ds \\
        & \qquad\quad
     := (R_{51})_j^n + (R_{52})_j^n.
   \end{split}
   \end{equation}
Then, define $\xi_j^n = u_j(t_n)-U_j^n = u_j^n - U_j^n$ for $0\leq n \leq N$ and $\Vec{\xi}^n=[\xi_1^n,\xi_2^n,\cdots,\xi_{J-1}^n]^{\top}$. By subtracting \eqref{eq3.2}-\eqref{eq3.4} from \eqref{eq3.18}-\eqref{eq3.20}, we arrive at the following error equations for $n\geq 2$,
\begin{equation}\label{eq3.22}
   \begin{split}
        \delta_t^2 \xi_j^n &+ G\left( \|\Vec{U}_{x\Bar{x}}^n\|^2 \right) \delta_t \xi_j^n  + \mu_0 (\xi_{j}^n)_{xx\Bar{x}\Bar{x}} +  \sum\limits_{p=1}^{n}w_{np} \left(\delta_t \xi_{j}^p \right)_{xx\Bar{x}\Bar{x}} \\
        & =  \left[  G\left( \|\Vec{U}_{x\Bar{x}}^n\|^2 \right) -  G\left(
        \int_0^1|u_{xx}(x,t_n)|^2 dx \right)  \right]\delta_tu_j^n + \sum\limits_{m=1}^{5}(R_m)_j^n,
   \end{split}
   \end{equation}
\begin{equation}\label{eq3.23}
   \begin{split}
        & \xi_0^n = \xi_{J}^n = 0, \quad
        (\xi_0^n)_{x\Bar{x}} = (\xi_J^n)_{x\Bar{x}} = 0, \quad  1\leq n \leq N,
   \end{split}
   \end{equation}
\begin{equation}\label{eq3.24}
   \begin{split}
        \xi_j^0 = 0, \quad \delta_t \xi_j^1 = \delta_tu_j^1 - (u_t)_j^0 = \frac{1}{\Delta t} \int_{0}^{t_1} u_{tt}(x_j,s)(t_1-s)ds, \quad 1\leq j \leq J-1.
   \end{split}
   \end{equation}
Then, we present the following convergence result.

\begin{theorem} \label{theorem3.2} Let the assumptions ($\mathbf{S1}$)-($\mathbf{S2}$) hold.
Let $\beta(t)$ be denoted by \eqref{eq1.5} or \eqref{eq1.6}, $\Vec{U}^n=[U_1^n,U_2^n,\cdots,U_{J-1}^n]^{\top}$ be the solution of the fully discrete scheme \eqref{eq3.2}-\eqref{eq3.4}, and $\Vec{u}^n=[u_1^n,u_2^n,\cdots,u_{J-1}^n]^{\top}$ be the solution of \eqref{eq3.18}-\eqref{eq3.20}, respectively. If satisfying \textbf{Assumption A},
then for $T<\infty$, it holds for $1\leq n \leq N$
\begin{equation}\label{qqq1}
 \begin{split}
    \sqrt{\|\delta_t(\Vec{u}^n-\Vec{U}^n)\|^2
       + \|\Vec{u}_{x\Bar{x}}^n -\Vec{U}_{x\Bar{x}}^n \|^2}  \leq C(T) (\Delta t + h^2).
 \end{split}
\end{equation}
In addition, if $h$ is small enough and $T<\infty$, we have
\begin{equation}\label{qqq2}
 \begin{split}
    \sqrt{\|\delta_t(\Vec{u}^n-\Vec{U}^n)\|^2
       + \|\Vec{u}^n -\Vec{U}^n \|^2}  \leq C(T) (\Delta t + h^2), \quad 1\leq n \leq N.
 \end{split}
\end{equation}
\end{theorem}
\begin{proof} First, taking the inner product of \eqref{eq3.22} with $\delta_t\Vec{\xi}^n$, then applying the assumption ($\mathbf{S1}$) and using \eqref{eq2.4} and \eqref{eq3.7}, we obtain
\begin{equation}\label{eq3.25}
   \begin{split}
       \frac{1}{2}\delta_t \|\delta_t \Vec{\xi}^n\|^2 &+ g_0 \|\delta_t \Vec{\xi}^n\|^2  + \frac{\mu_0}{2}\delta_t\|\Vec{\xi}^n_{x\Bar{x}}\|^2+ \sum\limits_{p=1}^{n}w_{np}\langle \delta_t\Vec{\xi}^p_{x\Bar{x}}, \delta_t\Vec{\xi}^n_{x\Bar{x}} \rangle \\
        & \leq   \left[  G\left( \|\Vec{U}_{x\Bar{x}}^n\|^2 \right) -  G\left(
        \int_0^1|u_{xx}(x,t_n)|^2 dx \right)  \right] \langle\delta_t\Vec{u}^n, \delta_t\Vec{\xi}^n\rangle \\
        & + \sum\limits_{m=1}^{5}\langle(\Vec{R}_m)^n, \delta_t\Vec{\xi}^n\rangle, \quad n\geq 2.
   \end{split}
   \end{equation}
First, we estimate the first term of the right-hand side of \eqref{eq3.25}. Similar to the analysis of \eqref{eq2.33}, then we have
\begin{equation}\label{eq3.26}
   \begin{split}
         \left| G\left(\|\Vec{u}_{x\Bar{x}}(t_n)\|^2 \right) - G\left(\|\Vec{U}_{x\Bar{x}}^n\|^2 \right) \right|  & \leq   C\|\Vec{\xi}_{x\Bar{x}}^n\|.
   \end{split}
   \end{equation}
Then, based on \eqref{eq2.30}-\eqref{eq2.32} and \eqref{eq3.26}, we get
\begin{equation}\label{eq3.27}
   \begin{split}
         \left|  G\left( \|\Vec{U}_{x\Bar{x}}^n\|^2 \right) -  G\left(
        \int_0^1|u_{xx}(x,t_n)|^2 dx \right) \right| \leq C\left( h^2 + \|\Vec{\xi}_{x\Bar{x}}^n\| \right).
   \end{split}
   \end{equation}
Now, summing \eqref{eq3.25} regarding $n$ from 2 to $N$, then multiplying $\Delta t$ and applying the Cauchy-Schwarz inequality, we obtain by \eqref{eq3.11},
\begin{equation}\label{eq3.28}
   \begin{split}
       & \frac{1}{2}  \|\delta_t\Vec{\xi}^N\|^2 + g_0 \Delta t\sum\limits_{n=2}^{N} \|\delta_t \Vec{\xi}^n\|^2  + \frac{\mu_0}{2}\|\Vec{\xi}^N_{x\Bar{x}}\|^2 \\
        & \leq   C\Delta t\sum\limits_{n=2}^{N}\left( h^2 + \|\Vec{\xi}_{x\Bar{x}}^n\| \right)  \|\delta_t\Vec{u}^n\|\| \delta_t\Vec{\xi}^n\| + \Delta t\sum\limits_{n=2}^{N}\sum\limits_{m=1}^{5}\|(\Vec{R}_m)^n \|\| \delta_t\Vec{\xi}^n\| \\
        & + \frac{1}{2}  \|\delta_t\Vec{\xi}^1\|^2+ \frac{\mu_0}{2}\|\Vec{\xi}^1_{x\Bar{x}}\|^2 +
        (\Delta t)w_{11}\|\delta_t\Vec{\xi}^1_{x\Bar{x}}\|^2.
   \end{split}
   \end{equation}
Using \textbf{Assumption A} and Taylor expansion with integral remainder, we have
\begin{equation}\label{eq3.29}
   \begin{split}
        \left|\delta_tu^n_j\right| =
        \left|u_t(x_j,t_n) + \frac{1}{\Delta t}\int_{t_{n-1}}^{t_n}u_{tt}(x_j,s)(t_{n-1}-s)ds \right| \leq C
        \Rightarrow \|\delta_t \Vec{u}^n\| \leq C.
   \end{split}
   \end{equation}
Combining \eqref{eq3.28} and \eqref{eq3.29}, then we get
\begin{equation}\label{eq3.30}
   \begin{split}
        \frac{\mu_0}{2}  \|\delta_t\Vec{\xi}^N\|^2& + \frac{\mu_0}{2}\|\Vec{\xi}^N_{x\Bar{x}}\|^2  \leq   C\Delta t\sum\limits_{n=2}^{N}\left( h^2 + \|\Vec{\xi}_{x\Bar{x}}^n\| \right)  \| \delta_t\Vec{\xi}^n\| \\
        & + \Delta t\sum\limits_{n=2}^{N}\sum\limits_{m=1}^{5}\|(\Vec{R}_m)^n \|\| \delta_t\Vec{\xi}^n\| + \frac{1}{2}  \left(\|\delta_t\Vec{\xi}^1\|^2+\|\Vec{\xi}^1_{x\Bar{x}}\|^2 \right).
   \end{split}
   \end{equation}
Denote a new norm
\begin{equation}\label{eq3.31}
   \begin{split}
        \|\Vec{\xi}^n\|_A : = \sqrt{\|\delta_t\Vec{\xi}^n\|^2+\|\Vec{\xi}^n_{x\Bar{x}}\|^2}, \quad 1\leq n \leq N.
   \end{split}
   \end{equation}
Then, we rewrite \eqref{eq3.30} as
\begin{equation*}
   \begin{split}
        \|\Vec{\xi}^N\|_A^2
             & \leq   C\Delta t\sum\limits_{n=2}^{N}\left( h^2 + \|\Vec{\xi}^n\|_A \right)  \|\Vec{\xi}^n\|_A \\
        & + \frac{2}{\mu_0}\Delta t\sum\limits_{n=2}^{N}\sum\limits_{m=1}^{5}\|(\Vec{R}_m)^n \|\|\Vec{\xi}^n\|_A + \frac{1}{\mu_0}  \|\Vec{\xi}^1\|_A^2, \quad N\geq 1.
   \end{split}
   \end{equation*}
By taking an appropriate $M$ so that $\|\Vec{\xi}^M\|_A=\max\limits_{1\leq n \leq N}\|\Vec{\xi}^n\|_A$, we have
\begin{equation}\label{eq3.33}
   \begin{split}
        \|\Vec{\xi}^M\|_A
             & \leq   C\Delta t\sum\limits_{n=2}^{M}\left( h^2 + \|\Vec{\xi}^n\|_A \right)  + \frac{2}{\mu_0}\Delta t\sum\limits_{n=2}^{M}\sum\limits_{m=1}^{5}\|(\Vec{R}_m)^n \| + \frac{1}{\mu_0}  \|\Vec{\xi}^1\|_A \\
             & \leq  C\Delta t\sum\limits_{n=2}^{N}\left( h^2 + \|\Vec{\xi}^n\|_A \right)  + \frac{2}{\mu_0} \Delta t\sum\limits_{n=2}^{N}\sum\limits_{m=1}^{5}\|(\Vec{R}_m)^n \| + \frac{1}{\mu_0}  \|\Vec{\xi}^1\|_A.
   \end{split}
   \end{equation}
Subsequently, we apply \eqref{eq3.24}, \textbf{Assumption A} and $\Vec{\xi}^0=0$ to obtain
\begin{equation*}
   \begin{split}
         \xi_j^1 = \int_0^{t_1}u_{tt}(x_j,s)(t_1-s)ds \Rightarrow \|\delta_t\Vec{\xi}^1\| \leq \Delta t \max\limits_{0\leq s \leq \Delta t}\|\Vec{u}_{tt}(s)\|,
   \end{split}
   \end{equation*}
thus, if $|u_{ttxxxx}(x,t)|\leq C$, we can similarly get
\begin{equation}\label{eq3.34}
   \begin{split}
         \|\Vec{\xi}^1_{x\Bar{x}}\| \leq (\Delta t)^2 \max\limits_{0\leq s \leq \Delta t}\|(\Vec{u}_{tt}(s))_{x\Bar{x}}\| \leq C(\Delta t)^2 \Rightarrow  \|\Vec{\xi}^1\|_A \leq C\Delta t.
   \end{split}
   \end{equation}
Furthermore, from \eqref{eq3.21}, with $|u_{txxxxxx}(x,t)|\leq C$, we have
\begin{equation}\label{eq3.35}
   \begin{split}
          \Delta t\sum\limits_{n=2}^{N}\sum\limits_{m=1}^{2}\|(\Vec{R}_m)^n \| \leq C(T)\left(C+ \int_0^T\beta(t)dt \right)h^2.
   \end{split}
   \end{equation}
Then, use \textbf{Assumption A} and Taylor expansion with the integral remainder to get
\begin{equation}\label{eq3.36}
   \begin{split}
          \Delta t\sum\limits_{n=2}^{N}\|(\Vec{R}_3)^n \|
          \leq \Delta t\sum\limits_{n=2}^{N} \int_{t_{n-2}}^{t_n}\|u_{ttt}(s)\|ds = \Delta t\int_{0}^{t_N}\|u_{ttt}(s)\|ds \leq C\Delta t.
   \end{split}
   \end{equation}
Employing the assumption ($\mathbf{S1}$) and \textbf{Assumption A} to yield
\begin{equation}\label{eq3.37}
   \begin{split}
          \Delta t\sum\limits_{n=2}^{N}\|(\Vec{R}_4)^n \|
          \leq g_1\Delta t\sum\limits_{n=2}^{N} \int_{t_{n-1}}^{t_n}\|u_{tt}(s)\|ds = g_1 \Delta t\int_{\Delta t}^{t_N}\|u_{tt}(s)\|ds \leq C\Delta t.
   \end{split}
   \end{equation}
Then for $(\Vec{R}_5)^n $, noting that
\begin{equation*}
   \begin{split}
    |(R_{51})^n_j| & = \left| \sum\limits_{p=1}^n \left\{ \int_{t_{p-1}}^{t_p}K(t_n-s)\left[ (u_t(x_j,s))_{xxxx} - (\delta_t u_j^p)_{xxxx}  \right] ds\right\} \right| \\
   & \leq \sum\limits_{p=1}^n \left\{ \int_{t_{p-1}}^{t_p}K(t_n-s)\left[ \int_{t_{p-1}}^{t_p}|u_{ttxxxx}(x_j,\theta)|d\theta  \right] ds\right\},
   \end{split}
   \end{equation*}
thence with $K(t)\leq C_0$, if $|u_{ttxxxx}(x,t)|\leq C$, we further obtain that
\begin{equation}\label{eq3.38}
   \begin{split}
          \Delta t\sum\limits_{n=2}^{N}\|(\Vec{R}_{51})^n \|
          & \leq C(T) \sup\limits_{0\leq t \leq T}\|u_{ttxxxx}(t)\| \Delta t.
   \end{split}
   \end{equation}
Besides, we rewrite $(\Vec{R}_{52})^n$ as
\begin{equation*}
   \begin{split}
    (R_{52})^n_j  = \frac{1}{\Delta t}\int_{t_{n-1}}^{t_n} \sum\limits_{p=1}^n  (\delta_t u_j^p)_{xxxx} \left[ \int_{t_{p-1}}^{t_p}K(t_n-s)ds - \int_{t_{p-1}}^{\min(t, t_p)}K(t-s)ds \right]dt,
   \end{split}
   \end{equation*}
where $|(\delta_t u_j^p)_{xxxx}|\leq \sup\limits_{0\leq t \leq T}\|u_{ttxxxx}(t)\|$. Thus, for $1\leq p \leq n-1$ and $t\in (t_{n-1},t_n)$, we have
\begin{equation*}
   \begin{split}
      \left| \int_{t_{p-1}}^{t_p}[K(t_n-s)-K(t-s)]ds \right| &\leq \int_{t_{p-1}}^{t_p} \int_t^{t_n}|K'(\vartheta-s)|d\vartheta ds \\
      & = \int_{t_{p-1}}^{t_p} \int_t^{t_n}|\beta (\vartheta-s)|d\vartheta ds,
   \end{split}
   \end{equation*}
which implies that
\begin{equation*}
   \begin{split}
      \sum\limits_{p=1}^{n-1}  \left| \int_{t_{p-1}}^{t_p}|K(t_n-s)-K(t-s)|ds \right| &\leq C \int_t^{t_n} \int_{0}^{t_{n-1}} (\vartheta-s)^{\alpha-1}ds d\vartheta \\
      & \leq C\int_t^{t_n} \left[ \vartheta^{\alpha} - (\vartheta-t_{n-1})^{\alpha} \right]  d\vartheta  \leq C \Delta t.
   \end{split}
   \end{equation*}
Also, for $p=n$ and $t\in [t_{n-1},t_n]$, we have
\begin{equation*}
   \begin{split}
     \left| \int_{t_{n-1}}^{t_n}K(t_n-s)ds - \int_{t_{n-1}}^{t}K(t-s)ds \right| \leq C \Delta t.
   \end{split}
   \end{equation*}
Then, following from above analyses,
\begin{equation}\label{eq3.38.1}
   \begin{split}
          \Delta t\sum\limits_{n=2}^{N}\|(\Vec{R}_{52})^n \|
          & \leq C(T) \sup\limits_{0\leq t \leq T}\|u_{ttxxxx}(t)\| \Delta t.
   \end{split}
   \end{equation}
Substituting \eqref{eq3.34}, \eqref{eq3.35}, \eqref{eq3.36}, \eqref{eq3.37}, \eqref{eq3.38} and \eqref{eq3.38.1} into \eqref{eq3.33}, we have
\begin{equation*}
   \begin{split}
        \|\Vec{\xi}^N\|_A
             & \leq   C\Delta t\sum\limits_{n=2}^{N}\left( h^2 + \|\Vec{\xi}^n\|_A \right)  +  C(T)(\Delta t + h^2).
   \end{split}
   \end{equation*}
Then, the discrete Gr\"{o}nwall's lemma yields \eqref{qqq1}.

\vskip 1mm
    Noting that Theorem \ref{theorem3.2} shows the convergence in the $H^2$ norm of the fully discrete difference scheme. However, we want to yield the convergence in the $L^2$ norm. Following from \cite[Lemma 2]{Omrani} and \cite[Lemma 4.7]{Xu3}, we have
\begin{equation*}
   \begin{split}
         \frac{\sin(\pi h/2)}{h/2}\|\Vec{U}_{\Bar{x}}\| \leq \|\Vec{U}_{x\Bar{x}}\|, \quad \text{if} \quad U_0 = U_J = 0.
   \end{split}
   \end{equation*}
Then, from \cite[Lemma 2]{Omrani}, we get the discrete Poincar\'{e} inequality
\begin{equation*}
   \begin{split}
         \frac{\sin(\pi h/2)}{h/2}\|\Vec{U}\| \leq \|\Vec{U}_{\Bar{x}}\|, \quad \text{if} \quad U_0 = U_J = 0.
   \end{split}
   \end{equation*}
Thus, when $h$ is small enough, we have $\frac{1}{\pi^2}\|\Vec{U}\| \leq \|\Vec{U}_{x\Bar{x}}\|$, which combines \eqref{qqq1} to yield \eqref{qqq2}. This finishes the proof.
\end{proof}

\subsection{Uniqueness of numerical solutions}

Here, the uniqueness of numerical solutions of \eqref{eq3.2} will be deduced. We give the following theorem.

\vskip 1mm
\begin{theorem}\label{theorem3.4} Let the conditions in Theorem \ref{theorem3.3} hold.
    If $T<\infty$ and $\Delta t$ small enough, then the fully discrete difference scheme \eqref{eq3.2}-\eqref{eq3.4} possesses a unique solution.
\end{theorem}

\begin{proof} Let $\Vec{U}^n\in \mathbb{R}^{J-1}$ and $\Vec{U}_{*}^n\in \mathbb{R}^{J-1}$ be the solutions of \eqref{eq3.2}, respectively, which satisfy $\Vec{U}_{*}^p = \Vec{U}^p$, $p=0,1$. Then, we assume that $\Vec{U}_{*}^m = \Vec{U}^m$ holds for $m=0,1,\cdots,N-1$. Below, we only need to prove that $\Vec{U}_{*}^N = \Vec{U}^N$ for \eqref{eq3.2}.

Define $\Vec{\xi}_{*}^n=\Vec{U}^n-\Vec{U}_{*}^n$ for $n=0,1,\cdots,N$. Similar to \eqref{eq3.22}, we also have
\begin{equation}\label{eq3.41}
   \begin{split}
        \delta_t^2 \Vec{\xi}_{*}^n &+ G\left( \|\Vec{U}_{x\Bar{x}}^n\|^2 \right) \delta_t \Vec{\xi}_{*}^n  + \mu_0(\Vec{\xi}_{*}^n)_{xx\Bar{x}\Bar{x}}
         + \sum\limits_{p=1}^{n}w_{np} (\Vec{\xi}_{*}^n )_{xx\Bar{x}\Bar{x}} \\
        & =  \left[  G\left( \|\Vec{U}_{*,x\Bar{x}}^n\|^2 \right) -  G\left(
        \|\Vec{U}_{x\Bar{x}}^n\|^2  \right)  \right]\delta_t \Vec{U}_*^n.
   \end{split}
   \end{equation}
Taking the inner product of \eqref{eq3.41} with $\delta_t\Vec{\xi}_*^n$, then using the assumption ($\mathbf{S1}$), \eqref{eq2.4} and \eqref{eq3.7}, we get
\begin{equation}\label{eq3.42}
   \begin{split}
       \frac{1}{2}\delta_t \|\delta_t \Vec{\xi}_*^n\|^2 &+ g_0 \|\delta_t \Vec{\xi}_*^n\|^2  + \frac{\mu_0}{2}\delta_t\|\Vec{\xi}^n_{*,x\Bar{x}}\|^2+ \sum\limits_{p=1}^{n}w_{np}\langle \delta_t\Vec{\xi}^p_{*,x\Bar{x}}, \delta_t\Vec{\xi}^n_{*,x\Bar{x}} \rangle \\
        & \leq   \left[  G\left( \|\Vec{U}_{*,x\Bar{x}}^n\|^2 \right) -  G\left(
       \|\Vec{U}_{x\Bar{x}}^n\|^2  \right)  \right] \langle\delta_t\Vec{U}_*^n, \delta_t\Vec{\xi}_*^n\rangle.
   \end{split}
   \end{equation}
Integrating \eqref{eq3.42} on $n$ from 2 to $N$, then multiplying $\Delta t$ and noting $\Vec{\xi}_{*}^0=\Vec{\xi}_{*}^1=0$,
\begin{equation}\label{eq3.43}
   \begin{split}
       \frac{1}{2} \|\delta_t \Vec{\xi}_*^N\|^2 & + \frac{\mu_0}{2}\|\Vec{\xi}^N_{*,x\Bar{x}}\|^2+\Delta t\sum\limits_{n=1}^{N}\sum\limits_{p=1}^{n}w_{np}\langle \delta_t\Vec{\xi}^p_{*,x\Bar{x}}, \delta_t\Vec{\xi}^n_{*,x\Bar{x}} \rangle \\
        & \leq \Delta t \sum\limits_{n=2}^{N}   \left|  G\left( \|\Vec{U}_{*,x\Bar{x}}^n\|^2 \right) -  G\left(
       \|\Vec{U}_{x\Bar{x}}^n\|^2  \right)  \right| \|\delta_t\Vec{U}_*^n\|\| \delta_t\Vec{\xi}_*^n\| : = \Theta(\Vec{U},\Vec{\xi}).
   \end{split}
   \end{equation}
Then, using Theorem \ref{theorem3.1}, we have $\|\delta_t\Vec{U}_*^n\|\leq C$, and applying the assumption ($\mathbf{S2}$), we get that
   \begin{equation}\label{eq3.44}
   \begin{split}
       \Theta(\Vec{U},\Vec{\xi})
       \leq C \Delta t \sum\limits_{n=2}^{N}\| \Vec{\xi}_{*,x\Bar{x}}^n\| \| \delta_t\Vec{\xi}_*^n\|  \leq \frac{C}{2}\Delta t \sum\limits_{n=2}^{N} \| \delta_t\Vec{\xi}_*^n \|^2 +  \frac{C}{2}\Delta t \sum\limits_{n=2}^{N} \| \Vec{\xi}_{*,x\Bar{x}}^n\|^2.
   \end{split}
   \end{equation}
Putting \eqref{eq3.44} into \eqref{eq3.43}, and then applying the discrete Gr\"{o}nwall's lemma, we obtain $\|\Vec{\xi}_{*}^N\|^2 \leq 0$ with $T<\infty$. The proof is finished.
\end{proof}

\section{Numerical simulation}\label{sec5}

In this section, we give two numerical examples to validate the effectiveness of the fully discrete difference scheme \eqref{eq3.2}-\eqref{eq3.4} and the correctness of the theoretical analysis. Since the scheme is implicit and nonlinear, we below compute and implement it by a fixed-point iterative algorithm, see \cite{Qiu}.

Let $U_j^n$ be the numerical solution, and then, with the unknown exact solutions, we denote the corresponding errors and the time convergence orders
\begin{equation*}
 E_{2}(\Delta t,h) = \sqrt{ h\sum_{j=1}^{J-1} \left|U_j^N(\Delta t,h)-U_j^{2N}(\Delta t /2,h)\right|^2 }, \quad {\rm rate}^t = \log_{2} \left(\frac{E_{2}(2\Delta t,h)}{E_{2}(\Delta t,h)}\right),
\end{equation*}
and denote the corresponding  errors and the space convergence orders
\begin{equation*}
 F_{2}(\Delta t,h) = \sqrt{ h\sum_{j=1}^{J-1}  \left|U_j^N(\Delta t,h)-U_{2j}^{N}(\Delta t,h/2)\right|^2}, \quad {\rm rate}^x = \log_{2} \left(\frac{F_{2}(\Delta t,2h)}{F_{2}(\Delta t,h)}\right).
\end{equation*}

\vskip 1mm
\textbf{Example 1.} In the first example, we select $T=1$, $G(v)=1+v$ and $\beta(t)$ as \eqref{eq1.5}. Let $u_0(x)=\sin(\pi x)$, $u_1(x)=\sin(2 \pi x)$ and $f(x,t)=e^{-\sigma t}t^{\alpha}\sin(\pi x)$ with the zero boundary conditions.

\vskip 1mm
First, in Table \ref{tab1} we fix the parameters $\sigma=\frac{6}{5}$, $\alpha=\frac{1}{2}$ and $J=32$, then test the errors and convergence orders in the time direction with different $\gamma$, from the results of Table \ref{tab1}, we see that the scheme possesses first-order accuracy in time. Then, fixing $\sigma=2$ and $N=64$ in Table \ref{tab2}, we test the errors and convergence orders in space, and the results show that the scheme has spatial second-order accuracy. These are consistent with the theoretical analysis, see Theorems \ref{theorem2.2} and \ref{theorem3.2}.

\begin{table}
    \center \footnotesize
    \caption{Errors and convergence orders in time by fixing $\sigma=\frac{6}{5}$, $\alpha=\frac{1}{2}$ and $J=32$ for Example 1.} \label{tab1}
    \begin{tabular}{cccccccccc}
      \toprule
    & \multicolumn{2}{c}{$\gamma=0$} & &\multicolumn{2}{c}{$\gamma=0.5$}&
     &\multicolumn{2}{c}{$\gamma=1$}\\
   \cmidrule{2-3}  \cmidrule{5-6} \cmidrule{8-9}
       $N$  & $E_{2}(\Delta t,h)$ & ${\rm rate}^t$  &  & $E_{2}(\Delta t,h)$ & ${\rm rate}^t$ &  & $E_{2}(\Delta t,h)$ & ${\rm rate}^t$\\
      \midrule
        16   &  $6.8176 \times 10^{-2}$  &    *   &  &  $6.9243 \times 10^{-2}$  &    *    &  & $7.1339 \times 10^{-2}$  &    *    \\
        32   &  $3.7830 \times 10^{-2}$  &  0.85  &  &  $3.8277 \times 10^{-2}$  &  0.86   &  & $3.8885 \times 10^{-2}$  &  0.88 \\
        64   &  $1.9723 \times 10^{-2}$  &  0.94  &  &  $1.9834 \times 10^{-2}$  &  0.95   &  & $1.9783 \times 10^{-2}$  &  0.97 \\
        128  &  $1.0068 \times 10^{-2}$  &  0.97  &  &  $1.0070 \times 10^{-2}$  &  0.98   &  & $9.8847 \times 10^{-3}$  &  1.00 \\
        256  &  $5.1218 \times 10^{-3}$  &  0.98  &  &  $5.1033 \times 10^{-3}$  &  0.98   &  & $4.9532 \times 10^{-3}$  &  1.00 \\
      \bottomrule
    \end{tabular}
\end{table}

\begin{table}
    \center \footnotesize
    \caption{Errors and convergence orders in space by fixing $\sigma=2$ and $N=64$ for Example 1.} \label{tab2}
     \resizebox{\textwidth}{!}{
   \begin{tabular}{ccccccccccc}
      \toprule
    && & \multicolumn{2}{c}{$\gamma=0$} & &\multicolumn{2}{c}{$\gamma=1$}&
     &\multicolumn{2}{c}{$\gamma=2$}\\
   \cmidrule{4-5}  \cmidrule{7-8} \cmidrule{10-11}
      $\alpha=\frac{1}{2}$ &$J$ & & $F_{2}(\Delta t,h)$ & ${\rm rate}^x$  & & $F_{2}(\Delta t,h)$ & ${\rm rate}^x$ & & $F_{2}(\Delta t,h)$ & ${\rm rate}^x$\\
      \midrule
       & 8  & &  $4.6196 \times 10^{-4}$  &    *    &  &  $5.9530 \times 10^{-4}$  &    *   &  &  $1.3001 \times 10^{-3}$  &    *    \\
       & 16 & &  $1.0977 \times 10^{-4}$  &  2.07   &  &  $1.4746 \times 10^{-4}$  &  2.01  &  &  $3.3723 \times 10^{-4}$  &  1.95 \\
       & 32 & &  $2.6967 \times 10^{-5}$  &  2.03   &  &  $3.6609 \times 10^{-5}$  &  2.01  &  &  $8.4756 \times 10^{-5}$  &  1.99 \\
       & 64 & &  $6.7104 \times 10^{-6}$  &  2.01   &  &  $9.1337 \times 10^{-6}$  &  2.00  &  &  $2.1212 \times 10^{-5}$  &  2.00 \\
       \midrule
       && & \multicolumn{2}{c}{$\gamma=0$} & &\multicolumn{2}{c}{$\gamma=1$}&
     &\multicolumn{2}{c}{$\gamma=2$}\\
   \cmidrule{4-5}  \cmidrule{7-8} \cmidrule{10-11}
     $\alpha=1$ &$J$ & & $F_{2}(\Delta t,h)$ & ${\rm rate}^x$  & & $F_{2}(\Delta t,h)$ & ${\rm rate}^x$ & & $F_{2}(\Delta t,h)$ & ${\rm rate}^x$\\
      \midrule
       & 8  & &  $8.5616 \times 10^{-3}$  &    *    &  &  $8.3152 \times 10^{-3}$  &    *   &  &  $7.6480 \times 10^{-3}$  &    *    \\
       & 16 & &  $2.3280 \times 10^{-3}$  &  1.89   &  &  $2.2836 \times 10^{-3}$  &  1.86  &  &  $2.1632 \times 10^{-3}$  &  1.82 \\
       & 32 & &  $5.9075 \times 10^{-4}$  &  1.98   &  &  $5.8071 \times 10^{-4}$  &  1.98  &  &  $5.5346 \times 10^{-4}$  &  1.97 \\
       & 64 & &  $1.4819 \times 10^{-4}$  &  2.00   &  &  $1.4574 \times 10^{-4}$  &  1.99  &  &  $1.3911 \times 10^{-4}$  &  1.99 \\
      \bottomrule
    \end{tabular}
    }
\end{table}

\vskip 1mm
\textbf{Example 2.} Here we choose $T=1$, $G(v)=\sqrt{1+v}$ and $\beta(t)$ as \eqref{eq1.6}. Let $u_0(x)=x^2(1-x)^2$, $u_1(x)=x^3(1-x)^3$ and $f(x,t)=0$ with zero boundary conditions.

\vskip 1mm
In Table \ref{tab3} we fix the parameters $\alpha=\frac{1}{2}$ and $J=64$, then test the errors and temporal convergence orders with different $\sigma$, from numerical results of Table \ref{tab3}, we observe that the proposed scheme can reach first-order accuracy in the time direction. Subsequently, by fixing $N=64$, Table \ref{tab4} lists the errors and spatial convergence orders with different $\alpha$ and $\sigma$, and then the results illustrate that the scheme approximates spatial second order. These results with the rate $O(\Delta t + h^2)$ are in accordance with the theory.

\begin{table}
    \center \footnotesize
    \caption{Errors and temporal convergence orders by fixing $\alpha=\frac{1}{2}$ and $J=64$ for Example 2.} \label{tab3}
    \begin{tabular}{cccccccccccc}
      \toprule
     & \multicolumn{2}{c}{$\sigma=1.5$} & &\multicolumn{2}{c}{$\sigma=2$}& \\
     \cmidrule{2-3}  \cmidrule{5-6}
      $N$ & $E_{2}(\Delta t,h)$ & ${\rm rate}^t$ & $N$ & $E_{2}(\Delta t,h)$ & ${\rm rate}^t$ &   \\
      \midrule
       128  &  $1.5816 \times 10^{-3}$   &  *      & 128   &  $1.4340 \times 10^{-3}$  &  *      \\
       256  &  $9.9110 \times 10^{-4}$   &  0.80   & 256   &  $8.5286 \times 10^{-4}$  &  0.75    \\
       512  &  $4.8755 \times 10^{-4}$   &  0.90   &512    &  $4.6378 \times 10^{-4}$  &  0.88    \\
       1024 &  $2.5027 \times 10^{-4}$   &  0.96   &1024   &  $2.3987 \times 10^{-4}$  &  0.95    \\
     \midrule
     & \multicolumn{2}{c}{$\sigma=2.5$} & &\multicolumn{2}{c}{$\sigma=3$}& \\
     \cmidrule{2-3}  \cmidrule{5-6}
      $N$ & $E_{2}(\Delta t,h)$ & ${\rm rate}^t$  & $N$ & $E_{2}(\Delta t,h)$ & ${\rm rate}^t$ &   \\
      \midrule
       128  &  $1.3409 \times 10^{-3}$   &  *      & 128   &  $1.2944 \times 10^{-3}$  &  *      \\
       256  &  $8.2030 \times 10^{-4}$   &  0.71   & 256   &  $8.1102 \times 10^{-4}$  &  0.67    \\
       512  &  $4.5227 \times 10^{-4}$   &  0.86   &512    &  $4.5221 \times 10^{-4}$  &  0.84    \\
       1024 &  $2.3552 \times 10^{-4}$   &  0.94   &1024   &  $2.3682 \times 10^{-4}$  &  0.93    \\
      \bottomrule
    \end{tabular}
\end{table}

 \begin{table}
 \centering \footnotesize
     \caption{Errors and spatial convergence orders when $N=64$ with different $\sigma$ and $\alpha$ for Example 2.}
     \label{tab4}  
 \begin{tabular}{cccccccccccc}
    \toprule
             \multirow{2}{*}{$\sigma$}& \multirow{2}{*}{$J$} &\multicolumn{2}{c}{$\alpha=0.3$} &&  \multicolumn{2}{c}{$\alpha=0.7$}      \\
              \cmidrule{3-4} \cmidrule{6-7}
     &                 &  $F_{2}(\Delta t,h)$    & ${\rm rate}^x$    & & $F_{2}(\Delta t,h)$    & ${\rm rate}^x$    \\
      \midrule
   & 16   &  $1.1603 \times 10^{-4}$   &  *          & &  $1.4705 \times 10^{-4}$    &  *            \\
$\sigma=1.5$
   & 32   &  $3.0157 \times 10^{-5}$   &  1.94       & &  $3.1429 \times 10^{-5}$    &  2.23      \\
   & 64   &  $7.5083 \times 10^{-6}$   &  2.01       & &  $7.4072 \times 10^{-6}$    &  2.09       \\
   & 128  &  $1.8266 \times 10^{-6}$   &  2.04       & &  $1.7679 \times 10^{-6}$    &  2.07      \\
     \midrule
   & 16   &  $1.0040 \times 10^{-4}$   &  *          & &  $1.7507 \times 10^{-4}$    &  *            \\
$\sigma=3.0$
   & 32   &  $2.7010 \times 10^{-5}$   &  1.89       & &  $3.8881 \times 10^{-5}$    &  2.17      \\
   & 64   &  $6.8769 \times 10^{-6}$   &  1.97       & &  $9.3905 \times 10^{-6}$    &  2.05       \\
   & 128  &  $1.7312 \times 10^{-6}$   &  1.99       & &  $2.3169 \times 10^{-6}$    &  2.02      \\
     \bottomrule
     \end{tabular}
   \end{table}

\section{Concluding remarks}\label{sec6}

In this work, we considered and analyzed the numerical solutions of problem \eqref{eq1.1}-\eqref{eq1.4}. First, we constructed a spatial semi-discrete difference scheme and performed the long-time stability and convergence analysis. Then we formulated the fully discrete difference scheme and proved the long-time stability, convergence, existence, and uniqueness of numerical solutions. In our future work, we will further consider the temporal second-order finite difference scheme for solving the damping viscoelastic Euler-Bernoulli equations.


\vskip 3mm
\noindent \scriptsize \textbf{Funding} This work was partially supported by the National Natural Science Foundation of China (No. 12301555), the Taishan Scholars Program of Shandong Province (No. tsqn202306083), and the National Key R\&D Program of China (No. 2023YFA1008903), and the Postdoctoral Fellowship Program of CPSF (No. GZC20240938).

\vskip 3mm
\noindent \scriptsize \textbf{Data Availability} The datasets are available from the corresponding author upon reasonable request.

\section*{Declarations}
\noindent \scriptsize \textbf{Conflict of interest} The authors declare that they have no conflict of interest.


\end{document}